\documentclass[12pt, reqno]{amsart}
\usepackage{amsfonts,mathrsfs,amsmath,amssymb}

\def\1{\hbox{1\kern-.35em\hbox{1}}}

\newtheorem{theorem}{Theorem}[section]
\newtheorem*{theorem*}{Theorem}
\newtheorem{lem}[theorem]{Lemma}

\newtheorem*{proposition*}{Proposition}

\newtheorem{remark}[theorem]{Remark}

\numberwithin{equation}{section}

\newcommand{\bea}{\begin{eqnarray}}
\newcommand{\eea}{\end{eqnarray}}
\newcommand{\be}{\begin{eqnarray*}}
\newcommand{\ee}{\end{eqnarray*}}

\newcommand{\Z}{{\mathbb Z}}

\newcommand{\C}{{\mathbb C}}

\def\tsv{\widetilde{\mathfrak{sv}}}

\newcommand{\ad}{{\rm ad}}

\newcommand{\Hom}{{\rm Hom}}

\newcommand{\Der}{{\rm Der}}
\newcommand{\Aut}{{\rm Aut}}
\def\iv{\mathcal{I}_{\mathfrak{v}}}

\def\W{{\rm\bf W}}

\def\Wab{{\rm\bf W}(a,b)}
\def\Wgab{{\rm\bf W^g}(a,b)}

\def\Igab{{\rm\bf I^g}(a,b)}

\def\qed{\hfill\mbox{$\Box$}}
\def\epsi{\epsilon}
\def\a{\alpha}

\def\l{\lambda}

\def\si{\sigma}

\def\dis{\displaystyle}

\def\vs{\vspace*}

\def\Z{\mathbb{Z}}

\def\C{\mathbb{C}}

\def\vp{\varphi}

\numberwithin{equation}{section}
\begin{document}
%
%%%%%%%%%%%%%%%%%%%%%%%%%%%%%%%%%%%%%%%%%%%%%%%%%%%%%%%%
\centerline
{\bf Derivations and automorphism
groups of }\centerline{{\bf the original deformative  ${\bf  Schr\ddot{o}dinger}$-{\bf Virasoro} algebras}\,
\footnote{ This work is supported by the NSFC grant 11101269 and Supported in part by the Fundamental Research
Funds for the Central Universities.\\\indent Corresponding author: qfjiang@sjtu.edu.cn}}
\vspace{1em} \centerline
{\bf Qifen Jiang$^{\,1)}$ and Song Wang$^{\,2)}$}
\centerline{\footnotesize $^{\,1)}$\small\it Department of
Mathematics, Shanghai Jiaotong University, Shanghai 200240, China}
\centerline{\small\it E-mail: qfjiang@sjtu.edu.cn}
\centerline{\footnotesize $^{\,2)}$\small\it Department of Mathematics, TongJi University, Shanghai
200092, China}
\centerline{\small\it E-mail: wangsong025@sina.com}

\vspace{1.5em}
{\bf Abstract}
\ \ In this paper, we determine the
derivation algebra and the automorphism group of the original deformative  ${\rm  Schr\ddot{o}dinger}$-{\rm Virasoro} algebras which is
the semi-direct product Lie algebra of the Witt algebra and its
tensor density module ${\rm\bf I^g}(a,b)$.



%\maketitle

%\tableofcontents

%%%%%%%%%%%%%%%%%%%%%%%%%%%%%%%%%%%%%%%%%%%%%%%%%%%%%%%%%%%%%%%%
%
\section{\bf Introduction }
%
%%%%%%%%%%%%%%%%%%%%%%%%%%%%%%%%%%%%%%%%%%%%%%%%%%%%%%%%%%%%%%%%
\label{sub0-c-related}

The Witt algebra ${\bf W}$ is an infinite dimensional Lie algebra
over the complex field $\C$, with basis $\{L_m\,|\,m\in\Z\}$ and the
defining relations:
$$
[L_{m}, L_{n}]=(m-n)L_{m+n},\quad\text{for any}\  m,n\in\Z.
$$
The unique nontrivial one-dimensional central extension of {\bf W}
is the Virasoro algebra {\bf Vir}, which is closely related to
Kac-Moody algebras and plays an important role in 2D conformal field
theory. There exist  different generalizations of the classical
 Witt algebra and the Virasoro algebra, which have been studied by
many authors (see for example \cite{DZ,KR,PZ,SZ,SCY}, etc.). In
\cite{mathieu} and \cite{KR},  a class of representations ${\rm\bf
I}(a,b)=\bigoplus\limits_{m\in\Z}\C I_{m}$ for {\bf W} with two
complex parameters $a$  and $b$ were introduced. The action of ${\bf
W}$ on ${\rm\bf I}(a,b)$ is given by
$$
L_m\cdot I_n=-(n+a+bm)I_{m+n},\quad \forall m,n\in\Z.
$$
${\rm\bf I}(a,b)$ is the so called  tensor density module. In \cite{ZD}, the representations of $W(2,2)$ have been studied in terms of
vertex operators algebras. In
\cite{SCY}, they consider a generation of the Witt algebra {\bf W}:
$$\Wab=\bigoplus_{m\in\Z}\C L_m\bigoplus_{n\in\Z}\C I_n$$ and determine its structures.
$\Wab$ has been considered in the mathematical and physical literature in
\cite{OR}. In this paper, we consider the following generalization of the Lie algebra
$\Wab$(reference to \cite{SCY}).

Let $\Wgab=\bigoplus\limits_{m\in\Z}\C
L_m\bigoplus\limits_{n\in\Z}\C I_n\bigoplus\limits_{k\in\Z}\C
Y_{k+1/2}$ equipped with the following brackets:
\begin{eqnarray}\label{defi1}
&&[L_{m}, L_{n}]=(m-n)L_{m+n},\\
&&[L_{m},I_{n}]=-(n+a+bm)I_{m+n},\\
&&[L_{m},Y_{n+1/2}]=-(n+\frac{1-m+a+bm}{2})Y_{m+n+1/2},\\
&&[I_{m}, I_{n}]=0,\\
&&[Y_{m+1/2}, Y_{n+1/2}]=(m-n)I_{m+n+1},\\
\label{defi2}&&[I_{m}, Y_{n+1/2}]=0.
\end{eqnarray}
for all $m,n\in\Z$. Then $\Wgab$ is an infinite dimensional Lie
algebra over the complex field $\C$. The  Lie algebra $\Wgab$ is actually the same as $\mathcal{L}_{a,b}$ which is called the original deformative  ${\rm Schr\ddot{o}dinger}$-{\rm Virasoro} algebras by \cite{LS} up to isomorphism. In \cite{LS}, the second cohomology group of original deformative  ${\rm Schr\ddot{o}dinger}$-{\rm Virasoro} algebras were determined. It is clear that $\Wgab\simeq
{\rm \bf W}\ltimes {\rm\bf I^g}(a,b)$, where ${\rm \bf W}$ is the
Witt algebra and ${\rm\bf I^g}(a,b)=\bigoplus\limits_{m\in\Z}\C
I_m\bigoplus\limits_{n\in\Z}\C Y_{n+1/2}$ is an ideal of $\Wgab$.
 In this paper, our aim  is to determine  the derivation algebra and the automorphism group of
$\Wgab$.

%%%%%%%%%%%%%%%%%%%%%%%%%%%%%%%%%%%%%%%%%%%%%%%%%%%%%%%%%%%%%%%%%%%%%%%%%%%%%%%%%%%%%%%%%%%%%%%%%%%%%%
%
\section{\bf The derivation algebra of $\Wgab$}
%
%%%%%%%%%%%%%%%%%%%%%%%%%%%%%%%%%%%%%%%%%%%%%%%%%%%%%%%%%%%%%%%%%%%%%%%%%%%%%%%%%%%%%%%%%%%%%%%%%%
\label{sub6-c-related}

Let $G$ be a commutative group, $\mathfrak{g}=\bigoplus\limits_{g\in
G}\mathfrak{g}_{g}$ a $G$-graded Lie algebra. A
$\mathfrak{g}$-module $V$ is called $G$-graded, if
$$V=\bigoplus\limits_{g\in G}V_{g\in G},
\;\;\; \mathfrak{g}_{g}V_{h}\subseteq V_{g+h},\;\;\; \forall\;
g,h\in G.$$ Let $\mathfrak{g}$ be a Lie algebra and $V$  a
$\mathfrak{g}$-module. A linear map $D: \mathfrak{g}\longrightarrow
V$ is called a derivation, if for any $x, y\in \mathfrak{g}$,
$$D[x, y]=x.D(y)-y.D(x).$$
If there exists some $v\in V$ such that $D:x\mapsto x.v$, then $D$
is called an inner derivation. Denote by $\Der(\mathfrak{g}, V)$ the
vector space of all derivations, ${\rm Inn}(\mathfrak{g}, V)$ the vector
space of all inner derivations. Set
$${\rm H}^{1}(\mathfrak{g}, V) =\Der(\mathfrak{g}, V) /{\rm Inn}(\mathfrak{g}, V).$$
Denote by $\Der(\mathfrak{g})$ the derivation algebra of
$\mathfrak{g}$, ${\rm Inn}(\mathfrak{g})$  the vector space of all inner
derivations of $\mathfrak{g}$.

In this section, we shall determine all the derivations of the Lie
algebra $\Wgab$.

\begin{lem}\label{P2.1}
For the Lie algebra $\Wgab$ defined in (\ref{defi1})-(\ref{defi2}), we have\\
$(1)$\quad As Lie algebras, $\Wgab\simeq {\rm\bf W^g}(a+k,b)$ for
any $k\in\ 2Z$.\\
$(2)$\quad  $\Wgab$ is perfect.\\
$(3)$\quad  the center of $\Wgab$ is
$$
{\rm Cent}(\Wgab)=\begin{cases} &\C I_0,\quad\text{if}\ (a,b)=(0,0),\\
&0,\quad\text{otherwises}.
\end{cases}
$$
$(4)$\quad $\Wgab$ has a natural $\frac{1}{2}\Z$-grading defined by
$$\Wgab=\bigoplus\limits_{m\in\Z}\Wgab_{\frac{m}{2}}=(\bigoplus\limits_{m\in\Z}\Wgab_{m})\bigoplus(\bigoplus\limits_{m\in\Z}\Wgab_{m+\frac{1}{2}}),$$
where $\Wgab_{m}=\C L_{m}\bigoplus\C I_{m},\Wgab_{m+\frac{1}{2}}=\C
Y_{m+\frac{1}{2}}$.
\end{lem}
\begin{proof}
It is straightforward to prove the lemma by the definition of
$\Wgab$.
\end{proof}

\begin{remark}\label{R}
  By $(1)$ of Lemma \ref{P2.1}, we may assume that
$a=0$ or $a=1$, if $a\in\Z$.
\end{remark}

By Proposition 1.1 in \cite{F}, we have the following lemma.
\begin{lem}\label{L3.1}
$$\Der(\Wgab)=\bigoplus\limits_{n\in\Z}\Der(\Wgab)_{\frac{n}{2}},$$
where $\Der(\Wgab)_{\frac{n}{2}}\Wgab_{\frac{m}{2}}\subseteq
\Wgab_{\frac{m+n}{2}}$ for all $m,n\in\Z$.
\end{lem}

\qed
\begin{lem}\label{L3.2}  For any $0\neq n\in\Z, a\neq 1,$ we have
${\rm H}^{1}(\Wgab_{0},\Wgab_{\frac{n}{2}})=0.$
\end{lem}
\begin{proof}
We have to prove $${\rm H}^{1}(\Wgab_{0},\Wgab_{m})=0, \; \forall
m\in\Z, m\neq 0,$$ and $${\rm
H}^{1}(\Wgab_{0},\Wgab_{m+\frac{1}{2}})=0, \; \forall m\in\Z.$$
(1)For any nonzero integer $m$, let
$\vp:\Wgab_{0}\longmapsto\Wgab_{m}$ be a derivation, then we may
assume
$$\vp(L_{0})=a_{1}L_{m}+b_1I_{m}, \quad\quad\; \vp(I_{0})=a_{2}L_{m}+b_2I_{m},$$
where $a_{i}, b_{i}\in\C, i=1, 2$. Because
$\vp[L_{0},I_{0}]=[\vp(L_{0}), I_{0}]+[L_{0}, \vp(I_{0})]$, we get
\begin{equation}\label{q48}
a_2(m-a)=0, a_1(a+bm)+b_{2}m=0.\end{equation} Since $a=0$ or $a=1$ or $a\not\in\Z,$ we
have $b_2=-\frac{a_1(a+bm)}{m}$ and the following several cases. \\
Case1. $a=0$ or $a\not\in\Z$. We have $a_{2}=0$ and
$$\vp(L_{0})=a_{1}L_{m}+b_1I_{m}, \quad\quad\; \vp(I_{0})=-\frac{a_1(a+bm)}{m}I_{m}.$$
Set $E_{m}=-\dis\frac{a_{1}}{m}L_{m}-\dis\frac{b_1}{m+a}I_{m}$, we have
$\vp(L_{0})=[L_{0}, E_{m}]$ and $\vp(I_{0})=[I_{0}, E_{m}]$, which
suggests  that $\vp\in Inn(\Wgab_{0},\Wgab_{m})$. \\
Case2. $a=1.$ \\
Subcase2.1: $m\neq 1$. By (\ref{q48}), we have $a_2=0.$\\
 If $m\neq -1$. Set $E_{m}=-\dis\frac{a_{1}}{m}L_{m}-\dis\frac{b_1}{m+1}I_{m}$, we also have
 $$\vp(L_{0})=[L_{0}, E_{m}]\quad\quad\; \vp(I_{0})=[I_{0}, E_{m}].$$ So $\vp\in Inn(\Wgab_{0},\Wgab_{m})$.\\
 If $m=-1$. Set $E_{-1}=-a_{1}L_{-1}$, we have $$\vp(L_{0})=adE_{-1}(L_0)+b_1I_{-1}\quad\quad\; \vp(I_{0})=adE_{-1}(I_0).$$
Subcase2.2: $m=1$. By (\ref{q48}), we have $a_2\in\C$ and
$$\vp(L_{0})=a_{1}L_{1}+b_1I_{1}, \quad\quad\; \vp(I_{0})=a_2L_1-a_1(1+b)I_{1}.$$
Set $E_{1}=-a_{1}L_{1}-\dis\frac{b_1}{2}I_{1}$, we have $$\vp(L_{0})=ad(-E_{1})(L_0)\quad\quad\; \vp(I_{0})=ad(-E_{1})(I_0)+a_2L_{1}.$$
(2) For all
$m\in\Z$, let $\vp:\Wgab_{0}\longmapsto\Wgab_{m+\frac{1}{2}}$ be a
derivation, then we may assume
$$\vp(L_{0})=a_{1}Y_{m+\frac{1}{2}}, \quad\quad\; \vp(I_{0})=b_{1}Y_{m+\frac{1}{2}},$$
where $a_{1}, b_{1}\in\C$. Since $\vp[L_{0},I_{0}]=[\vp(L_{0}),
I_{0}]+[L_{0}, \vp(I_{0})]$, we have
$$\vp(-aI_0)=-ab_1Y_{m+\frac{1}{2}}=[L_{0},b_1Y_{m+\frac{1}{2}}]=-b_1(m+\frac{1+a}{2})Y_{m+\frac{1}{2}},$$
then \begin{equation}\label{q49} b_1[\frac{a}{2}-(m+\frac{1}{2})]Y_{m+\frac{1}{2}}=0.\end{equation}
Case1. $a=0$ or $a\not\in\Z$. By (\ref{q49}), we have $b_1=0$ and
$$\vp(L_{0})=a_{1}Y_{m+\frac{1}{2}}, \quad\quad\; \vp(I_{0})=0$$
Set
$E_{m+\frac{1}{2}}=\dis\frac{a_1}{m+\dis\frac{1+a}{2}}Y_{m+\frac{1}{2}}$, we have
$$\vp(L_{0})=a_{1}Y_{m+\frac{1}{2}}=[E_{m+\frac{1}{2}}, L_0], \quad
\vp(I_{0})=0=[E_{m+\frac{1}{2}}, I_0].$$ So $\vp\in
Inn(\Wgab_{0},\Wgab_{m+\frac{1}{2}})$. \\
Case2. $a=1$. By (\ref{q49}), if $m\neq 0$, we have $b_1=0$ and
$$\vp(L_{0})=a_{1}Y_{m+\frac{1}{2}}, \quad\quad\; \vp(I_{0})=0$$
If $m\neq -1$, set $E_{m+\frac{1}{2}}=\dis\frac{a_1}{m+1}Y_{m+\frac{1}{2}}$, we have
$$\vp(L_{0})=[E_{m+\frac{1}{2}}, L_0]=adE_{m+\frac{1}{2}}(L_0), \quad
\vp(I_{0})=0=adE_{m+\frac{1}{2}}(I_0).$$
If $m=-1$. By the relations of brackets for $\Wgab$, we know $\vp$ is an outer derivation.\\
By (\ref{q49}), if $m=0$, we have $b_1\in\C$ and
$$\vp(L_{0})=a_{1}Y_{\frac{1}{2}}, \quad\quad\; \vp(I_{0})=b_1Y_{\frac{1}{2}}.$$
Set $E_{\frac{1}{2}}=a_1Y_{\frac{1}{2}}$, we have
$$\vp(L_{0})=[E_{\frac{1}{2}}, L_0]=adE_{\frac{1}{2}}(L_0), \quad
\vp(I_{0})=adE_{\frac{1}{2}}(I_0)+b_1Y_{\frac{1}{2}}.$$This completes the proof of
lemma.
\end{proof}

\begin{lem}\label{L3.3}
$\Hom_{\Wgab_{0}}(\Wgab_{\frac{m}{2}}, \Wgab_{\frac{n}{2}})=0$ for
all $m,n\in\Z$, $m\neq n, a\neq 1.$
\end{lem}
\begin{proof}
We have to consider the following four
identities:\begin{eqnarray}\label{L1}\Hom_{\Wgab_{0}}(\Wgab_{m}, \Wgab_{n})=0, \ \ m\neq n, m, n\in\Z;\\
\label{L2}\Hom_{\Wgab_{0}}(\Wgab_{m}, \Wgab_{n+\frac{1}{2}})=0, \quad\quad\ m, n\in\Z;\\
\label{L3}\Hom_{\Wgab_{0}}(\Wgab_{m+\frac{1}{2}}, \Wgab_{n})=0, \quad\quad\ m, n\in\Z;\\
\label{L4}\Hom_{\Wgab_{0}}(\Wgab_{m+\frac{1}{2}}, \Wgab_{n+\frac{1}{2}})=0, m\neq n, m, n\in\Z.\end{eqnarray}

For integers $m\neq n$, let $f\in
\Hom_{\Wgab_{0}}(\Wgab_{\frac{m}{2}}, \Wgab_{\frac{n}{2}})$, then
for $E_0\in\Wgab_{0}, \\ E_{m}\in\Wgab_{\frac{m}{2}}$, we have
\begin{equation}\label{30}
f([E_{0}, E_{\frac{m}{2}}])=[E_{0}, f(E_{\frac{m}{2}})].
\end{equation}
(1)It follows from $f([L_{0}, L_{m}])=[L_{0}, f(L_{m})]$ that
$f(-mL_{m})=[L_{0}, f(L_{m})]$. Assume $f(L_{m})=a_1L_{n}+b_1I_{n}$,
we have $-ma_1L_{n}-mb_1I_{n}=-a_1nL_{n}-b_1nI_{n}-b_1aI_{n}$. Since $m\neq n$, we get
\begin{equation}\label{q50}a_1=0,\quad\quad\ (m-n-a)b_1=0.\end{equation}
For $a=0$ or $a\not\in\Z,$ we have $b_1=0$ and
$f(L_{m})=0$. On the other hand, we have $f([L_{0}, I_{m}])=[L_{0},
f(I_{m})]$. Similarly, we have  $f(I_{m})=0$. Therefore,
$\Hom_{\Wgab_{0}}(\Wgab_{m}, \Wgab_{n})=0$.\\
For $a=1$. By (\ref{q50}), we get $$b_1=0, \ \ \ m-n\neq 1; \ \ \ b_1\in\C, \ \ \ m-n=1.$$ So $f(L_m)=b_1I_{m-1}$. Consequently (\ref{L1}) holds.\\
(2)By $f([L_{0}, L_{m}])=[L_{0}, f(L_{m})]$ and $f([L_{0}, I_{m}])=[L_{0}, f(I_{m})]$ we have
\begin{eqnarray}\label{L5}m f(L_{m})=(n+\frac{1+a}{2})f(L_{m}),\\
\label{L6}(m+a)f(I_{m})=(n+\frac{1+a}{2})f(I_{m}).\end{eqnarray}
For $a=0$ or $a\not\in\Z$, obviously, $f(L_{m})=0$ and $f(I_{m})=0$.\\
For $a=1$, by (\ref{L5}) and (\ref{L6}), we have $$f(L_{m})=0, \ \ m-n\neq 1; \ \ \ f(I_{m})=0, \ \ m\neq n.$$
Then
$$\Hom_{\Wgab_{0}}(\Wgab_{m}, \Wgab_{n+\frac{1}{2}})=0,\ \ \
a\neq 1,$$which shows (\ref{L2}).\\
(3) By (\ref{30}), we have $f([L_{0},
Y_{m+\frac{1}{2}}])=[L_{0}, f(Y_{m+\frac{1}{2}})]$. Assume
$f(Y_{m+\frac{1}{2}})=x_{n}L_{n}+y_{n}I_{n}$, we have\begin{equation}\label{L7}
-(m+\frac{1+a}{2})(x_{n}L_{n}+y_{n}I_{n})=-x_{n}nL_{n}-y_{n}(n+a)I_{n}.\end{equation}
For $a=0$ or $a\not\in\Z$, we have $x_{n}=0, y_{n}=0$. Then $f(Y_{m+\frac{1}{2}})=0$. For $a=1$, by (\ref{L7}), we have
$$(m+1-n)x_{n}=0, \ \ \ (m-n)y_{n}=0.$$ So $$x_{n}=0, \ \ \ m-n\neq -1; \ \ \ y_{n}=0, \ \ \ m\neq n$$
Therefore $$\Hom_{\Wgab_{0}}(\Wgab_{m+\frac{1}{2}}, \Wgab_{n})=0,
a\neq 1.$$ This shows (\ref{L3}).\\
(4) Assume
$f(Y_{m+\frac{1}{2}})=x_{n}Y_{n+\frac{1}{2}}$, by $f([L_{0},
Y_{m+\frac{1}{2}}])=[L_{0}, f(Y_{m+\frac{1}{2}})]$, we have
$$(m+\frac{1+a}{2})f(Y_{m+\frac{1}{2}})=(n+\frac{1+a}{2})f(Y_{m+\frac{1}{2}}).$$
Since $m\neq n$, we get $f(Y_{m+\frac{1}{2}})=0.$ So
$$\Hom_{\Wgab_{0}}(\Wgab_{m+\frac{1}{2}}, \Wgab_{n+\frac{1}{2}})=0,
m\neq n.$$ This proves (\ref{L4}).
\end{proof}

By Lemma \ref{L3.2}-\ref{L3.3} and Proposition 1.2 in \cite{F}, we
have the following Lemma.

\begin{lem}\label{L3.4}\begin{eqnarray*}
\Der(\Wgab,\Wgab) & = & \Der(\Wgab)_{0}+Inn(\Wgab)+\Der(\Wgab)_{1}\\
& & {}+\Der(\Wgab)_{-1}+\Der(\Wgab)_{\frac{1}{2}}+\Der(\Wgab)_{-\frac{1}{2}}.
\end{eqnarray*}
\end{lem}

\begin{lem}\label{L3.6}
$(1)$\quad Up to isomorphism, for
$(a,b)\not\in\{(0,0),(0,1),(0,2)\}$, we have $${\rm H}^{1}(\Wgab_{\frac{m}{2}},
\Wgab_{\frac{m}{2}})=\C \bar{D},$$ where
$$\bar{D}(L_{m})=0, \ \ \ \bar{D}(I_{m})=I_{m},\ \ \ \ \bar{D}(Y_{m+\frac{1}{2}})=Y_{m+\frac{1}{2}},\ \ \  \forall\;m\in\Z.$$
\\
$(2)$\quad  ${\rm H}^{1}(\W^{\rm\bf g}(0, 0)_{\frac{m}{2}}, \W^{\rm\bf g}(0, 0)_{\frac{m}{2}})=\C \bar{D}_{1}\bigoplus \C
\bar{D}_{2}\bigoplus\C \bar{D}_{3},$ where
$$\bar{D}_{1}(L_{m})=m I_{m}, \ \ \ \bar{D}_{1}(I_{m})=0, \ \ \ \bar{D}_{1}(Y_{m+\frac{1}{2}})=0,$$
$$\bar{D}_{2}(L_{m})=(m-1)I_m,\ \ \bar{D}_{2}(I_{m})=0, \ \ \ \bar{D}_{2}(Y_{m+\frac{1}{2}})=0,$$
$$\bar{D}_{3}(L_{m})=0, \ \ \ \bar{D}_{3}(I_{m})=I_{m},\ \ \ \bar{D}_{3}(Y_{m+\frac{1}{2}})=Y_{m+\frac{1}{2}}, \ \ \ \forall\;m\in\Z.$$
\\
$(3)$\quad  ${\rm
H}^{1}(\W^{\rm\bf g}(0,1)_{\frac{m}{2}},\W^{\rm\bf g}(0,1)_{\frac{m}{2}})=\C\bar{D}_{1}\bigoplus\C\bar{D}_{2},$ where
$$\bar{D}_{1}(L_{m})=0, \ \ \bar{D}_{1}(I_{m})=I_{m}, \ \ \bar{D}_1(Y_{m+\frac{1}{2}})=Y_{m+\frac{1}{2}};$$
$$ \bar{D}_{2}(L_{m})=m(m-1)I_{m}, \ \  \bar{D}_{2}(I_{m})=0,\ \ \ \bar{D}_{2}(Y_{m+\frac{1}{2}})=0  \forall\;m\in\Z.$$
\\
$(4)$\quad  ${\rm
H}^{1}(\W^{\rm\bf g}(0,2)_{\frac{m}{2}},\W^{\rm\bf g}(0,2)_{\frac{m}{2}})=\C\bar{D}_{1}\bigoplus\C\bar{D}_{2},$ where
$$\bar{D}_{1}(L_{m})=0, \ \ \bar{D}_{1}(I_{m})=I_{m}, \ \ \bar{D}_1(Y_{m+\frac{1}{2}})=Y_{m+\frac{1}{2}};$$
$$ \bar{D}_{2}(L_{m})=m^{3}I_{m}, \ \  \bar{D}_{2}(I_{m})=0,\ \ \ \bar{D}_{2}(Y_{m+\frac{1}{2}})=0 \  \forall\; m\in\Z.$$
\end{lem}

\begin{proof}
For  $D\in \Der(\Wgab, \Wgab)_{0}$, assume that
$$D(L_{m})=a_{11}^mL_{m}+a_{12}^mI_m, \ \ D(I_{m})=a_{21}^mL_m+a_{22}^mI_{m},\ \ D(Y_{m+\frac{1}{2}})=b^{m+\frac{1}{2}}Y_{m+\frac{1}{2}},$$
for all $m\in\Z$, where $a_{ij}^m, b^{m+\frac{1}{2}}\in\C, i, j=1, 2$. By the definition of
derivation and the product in $\Wgab$, we get
\begin{equation}\label{32}
(m-n)a_{11}^{m+n}=(m-n)a_{11}^m+(m-n)a_{11}^n,
\end{equation}
\begin{equation}\label{q51}
(m-n)a_{12}^{m+n}=(m+a+bn)a_{12}^{m}-(n+a+bm)a_{12}^{n},
\end{equation}
\begin{equation}\label{q52}
-(n+a+bm)a_{21}^{m+n}=(m-n)a_{21}^{n},\end{equation}
\begin{equation}\label{q53}
(n+a+bm)a_{22}^{m+n}=(n+a+bm)a_{11}^{m}+(n+a+bm)a_{22}^{n},
\end{equation}
\begin{equation}\label{q54}
(n+\dis\frac{1-m+a+bm}{2})b^{m+n+\frac{1}{2}}=(n+\dis\frac{1-m+a+bm}{2})(a_{11}^m+b^{n+\frac{1}{2}}),
\end{equation}
\begin{equation}\label{q55}
(m+a+bn)a_{21}^{n}=(n+a+bm)a_{21}^{m},
\end{equation}
\begin{equation}\label{q56}
(n+\dis\frac{1-m+a+bm}{2})a_{21}^{m}=0,
\end{equation}
\begin{equation}\label{q57}
(m-n)a_{21}^{m+n+1}=0,
\end{equation}
\begin{equation}\label{34}
(m-n)a_{22}^{m+n+1}=(m-n)b^{m+\frac{1}{2}}+(m-n)b^{n+\frac{1}{2}},
\end{equation}
for all $m,n\in\Z$. Let $n=0$ in (\ref{32})-(\ref{q53}), we have
\begin{equation}\label{36}
a_{11}^0=0,\end{equation}
\begin{equation}\label{q58}
(a+bm)a_{12}^{0}=aa_{12}^{m},
\end{equation}
\begin{equation}\label{q59}
-(a+bm)a_{21}^{m}=ma_{21}^{0},
\end{equation}
\begin{equation}\label{38}
(a+bm)a_{22}^{m}=(a+bm)(a_{11}^m+a_{22}^{0}).
\end{equation}
On the other hand, let $m=0$ in (\ref{q52}) and (\ref{q56}), we get
\begin{equation}\label{q60}
aa_{21}^n=0, \ \ \ a_{21}^0(n+\frac{1+a}{2})=0
\end{equation}
It follows from (\ref{32}), (\ref{q57}) and (\ref{34}) that
\begin{equation}\label{q61}
a_{11}^{m+n}=a_{11}^{m}+a_{11}^{n}, \ \ \ m\neq n,
\end{equation}
$$a_{21}^{m+n+1}=0, \ \ \ m\neq n,$$
\begin{equation}\label{q62}
 a_{22}^{m+n+1}=b^{m+\frac{1}{2}}+b^{n+\frac{1}{2}}, \ \ \ m\neq n.
\end{equation}Let $m=-n\neq 0$ in (\ref{q61}) and combine (\ref{36}), we have
$$a_{11}^{-m}=-a_{11}^{m}, \ \ \forall m\in\Z.$$
On the other hand, let $n=1$ in (\ref{q61}), we have
\begin{equation}\label{q63}a_{11}^{m+1}=a_{11}^{m}+a_{11}^{1}, \ \ \ m\neq 1.\end{equation}
By induction on $m\in\Z^{+}$ in (\ref{q63}), we get
\begin{equation}\label{q64}a_{11}^{m}=a_{11}^{2}+(m-2)a_{11}^{1}, \ \ \ m\neq 1.\end{equation}
Let $m=4, n=-2$ in (\ref{q61}), we have\begin{equation}\label{q65}a_{11}^2=2a_{11}^1.\end{equation}
Combine (\ref{q64}) and (\ref{q65}), we get $$a_{11}^{m}=ma_{11}^{1}, \ \ \ \forall m\in\Z.$$
Let $m=0$ in (\ref{q57}), we have $a_{21}^{n+1}=0, \ \ n\neq 0$. Moreover let $n=-m\neq 0$ in (\ref{q57}), We also have
$a_{21}^{1}=0$. So we get $$a_{21}^{m}=0, \ \ \ \forall m\in\Z.$$

{\bf Case 1:}  $a\not\in\Z$.   By (\ref{q58}), we have
$a_{12}^{m}=\dis\frac{a+bm}{a}a_{12}^0$ for all $m\in\Z$.

{\bf Subcase 1.1:} If $bm+a\neq 0$ for all $m\in\Z$, By (\ref{38}), we have
$$a_{22}^{m}=a_{11}^{m}+a_{22}^{0}=ma_{11}^{1}+a_{22}^{0}, \ \ \forall m\in\Z$$ .
Let $n=0$ in (\ref{q62}), we have
\begin{equation}\label{q66}b^{m+\frac{1}{2}}=a_{22}^{m+1}-b^{\frac{1}{2}}=(m+1)a_{11}^{1}+a_{22}^{0}-b^{\frac{1}{2}}, \  m\neq 0.\end{equation}
On the other hand, let $n=0, m=1$ in (\ref{q54}) and use (\ref{q66}), we have
\begin{equation}\label{q67}b^{\frac{1}{2}}=b^{1+\frac{1}{2}}-a_{11}^{1}=a_{11}^1+a_{22}^0-b^{\frac{1}{2}}\end{equation}
So $$b^{\frac{1}{2}}=\frac{1}{2}(a_{11}^1+a_{22}^0).$$ Combine (\ref{q66}) and (\ref{q67}), we get
$$b^{m+\frac{1}{2}}=ma_{11}^{1}+b^{\frac{1}{2}}=(m+\frac{1}{2})a_{11}^{1}+\frac{1}{2}a_{22}^{0}, \ \ \forall m\in\Z.$$
Then
$$D(L_{m})=ma_{11}^{1}L_{m}+\dis\frac{a+bm}{a}a_{12}^0I_{m}, \ \ \ D(I_{m})=(ma_{11}^{1}+a_{22}^0)I_{m},$$
$$D(Y_{m+\frac{1}{2}})=[(m+\frac{1}{2})a_{11}^{1}+\frac{1}{2}a_{22}^{0}]Y_{m+\frac{1}{2}},$$
for all $m\in\Z$. Set $E_{0}=-a_{11}^1L_0+\dis\frac{a_{12}^0}{a}I_{0}$, then
$$D(L_{m})=ad E_{0}(L_{m}), \ \ D(I_{m})=ad E_{0}(I_{m})+(a_{22}^{0}-aa_{11}^{1})I_{m},$$
$$D(Y_{m+\frac{1}{2}})=ad E_{0}(Y_{m+\frac{1}{2}})+\frac{1}{2}(a_{22}^{0}-aa_{11}^{1})Y_{m+\frac{1}{2}},$$
for all $m\in\Z$. Let $\bar{D}\in {\rm H}^{1}(\Wgab_{\frac{m}{2}}, \Wgab_{\frac{m}{2}})$ such that
$$\bar{D}(L_{m})=0, \bar{D}(I_{m})=I_{m}, \bar{D}(Y_{m+\frac{1}{2}})=Y_{m+\frac{1}{2}}$$ for all $m\in\Z$, then we
have
$${\rm H}^{1}(\Wgab_{\frac{m}{2}}, \Wgab_{\frac{m}{2}})=\C \bar{D}.$$

{\bf Subcase 1.2:} If  there exists some $\mu\in\Z$ such that
$b\mu+a=0$, then it follows from (\ref{38}) that $$a_{22}^{m}=a_{11}^{m}+a_{22}^{0}=ma_{11}^{1}+a_{22}^{0}, \ \ \  m\neq \mu.$$ Obviously, $\mu\neq
0$, (or else $a=0$, a contradiction.) Letting $m=n=\mu$ in
(\ref{q53}), we get
$$a_{22}^{\mu}=a_{22}^{2\mu}-a_{11}^{\mu}=2\mu a_{11}^{1}+a_{22}^{0}-a_{11}^{\mu}=\mu a_{11}^{1}+a_{22}^{0}.$$
Therefore, we still have $a_{22}^{m}=ma_{11}^{1}+a_{22}^{0}$ for all $m\in\Z$. %Let $n=0$ in (\ref{q62}), we have
%$$b^{m+\frac{1}{2}}=a_{22}^{m+1}-b^{\frac{1}{2}}=(m+1)a_{11}^{1}+a_{22}^{0}-b^{\frac{1}{2}}, \ \ \ m\neq 0.$$
Let $m=\mu, n=-\mu$ in (\ref{q54}) and use (\ref{q66}), we have
%\begin{equation}\label{q68}
%(n+\dis\frac{1-\mu}{2})b^{\mu+n+\frac{1}{2}}=(n+\dis\frac{1-\mu}{2})(a_{11}^{\mu}+b^{n+\frac{1}{2}}).
%\end{equation}Let $$ in (\ref{q68}), we have
$$b^{\frac{1}{2}}=\frac{1}{2}(a_{11}^{1}+a_{22}^{0}).$$ According to (\ref{q66}), we still get $b^{m+\frac{1}{2}}=(m+\frac{1}{2})a_{11}^{1}+\frac{1}{2}a_{22}^0$ for all $m\in\Z$. Then we have
the same result as Subcase 1.1.

{\bf Case 2:}  $a=0$. From (\ref{q58}) we get $ba_{12}^0=0$. Obviously,
$a_{12}^0=0$ when $b\neq 0$. Let $n=-m$ in (\ref{q51}), then we have
\begin{equation}\label{39}
(1-b)[a_{12}^m+a_{12}^{-m}]=2a_{12}^0,\quad\quad\quad\; m\neq 0.
\end{equation}

{\bf Subcase 2.1:} $b=0$. From (\ref{q51}) and (\ref{q53}), we have
\begin{equation}\label{40}
(m-n)a_{12}^{m+n}=ma_{12}^m-na_{12}^n,
\end{equation}
\begin{equation}\label{41}
a_{22}^{m+n}=a_{11}^m+a_{22}^n, \ \ \ \ \ \ \ n\neq 0.
\end{equation}
Let $n=1$ in (\ref{40}) and (\ref{41}) respectively, we have
\begin{equation}\label{43}
(m-1)a_{12}^{m+1}=ma_{12}^m-a_{12}^1.\end{equation}
\begin{equation}
\label{q69}a_{22}^{m+1}=a_{11}^m+a_{22}^1=ma_{11}^1+a_{22}^1, \ \ \ \forall m\in\Z.\end{equation}
On the other hand, let $n=1, m=-1$ in (\ref{41}), we have
\begin{equation}\label{q70}a_{22}^{1}=a_{11}^{1}+a_{22}^0.\end{equation} From (\ref{q69}) and (\ref{q70}),
it is easy to deduce that $$a_{22}^{m}=ma_{11}^1+a_{22}^0, \ \ \ \forall m\in\Z.$$
 It follows from (\ref{39}) that
\begin{equation}\label{42}
a_{12}^{-m}=2a_{12}^0-a_{12}^m,\quad\quad\quad\; \forall\;m\in\Z.
\end{equation}
By induction on $m>1$ in (\ref{43}), we can deduce that
$$a_{12}^m=(m-1)a_{12}^2-(m-2)a_{12}^1, \ \ \ \ \ \ m>0.$$
Set $m=-2$ in (\ref{43}) and use (\ref{42}), then we get
$$a_{12}^0=2a_{12}^1-a_{12}^2.$$
Combine the two identities, we have
$$a_{12}^m=ma_{12}^1-(m-1)a_{12}^0, \ \ \ \forall m\in\Z.$$
Let $n=0, m=2$ in (\ref{q54}) and use (\ref{q66}), we have $$b^{\frac{1}{2}}=\frac{1}{2}(a_{11}^1+a_{22}^0).$$
So we get $$b^{m+\frac{1}{2}}=ma_{11}^{1}+b^{\frac{1}{2}}=(m+\frac{1}{2})a_{11}^{1}+\frac{1}{2}a_{22}^{0}, \ \ \forall m\in\Z.$$
Therefore,
$$D(L_{m})=ma_{11}^1L_m+[ma_{12}^1-(m-1)a_{12}^0]I_{m},$$
$$D(I_{m})=(a_{22}^0+ma_{11}^1)I_{m}, \ \ D(Y_{m+\frac{1}{2}})=[(m+\frac{1}{2})a_{11}^{1}+\frac{1}{2}a_{22}^{0}]Y_{m+\frac{1}{2}},$$
for all $m\in\Z$. Setting $E_0=-a_{11}^1L_0$, we can deduce that
$$D(L_{m})=adE_0(L_m)+[ma_{12}^1-(m-1)a_{12}^0]I_{m},$$
$$D(I_{m})=adE_0(I_m)+a_{22}^0I_{m}, \ \; D(Y_{m+\frac{1}{2}})=adE_0(Y_{m+\frac{1}{2}})+\frac{1}{2}a_{22}^{0}Y_{m+\frac{1}{2}}.$$
Let $\bar{D_{1}}, \bar{D_{2}}, \bar{D_{3}}\in {\rm
H}^{1}(\W^{\rm\bf g}(0,0)_{\frac{m}{2}}, \W^{\rm\bf g}(0,0)_{\frac{m}{2}})$ such that
$$\bar{D}_{1}(L_{m})=m I_{m},  \quad \quad \bar{D}_{1}(I_{m})=0, \quad \quad \bar{D}_1(Y_{m+\frac{1}{2}})=0;$$
$$\bar{D}_{2}(L_{m})=(m-1)I_{m}, \quad\quad \bar{D}_{2}(I_{m})=0, \quad \quad \bar{D}_2(Y_{m+\frac{1}{2}})=0;$$
$$\bar{D}_{3}(L_{m})=0, \quad \quad \quad  \bar{D}_{3}(I_{m})=I_{m}, \quad \quad \bar{D}_3(Y_{m+\frac{1}{2}})=Y_{m+\frac{1}{2}},$$
for all $m\in\Z$, then they are all outer derivations and
$${\rm H}^{1}(\W^{\rm\bf g}(0,0)_{\frac{m}{2}}, \W^{\rm\bf g}(0,0)_{\frac{m}{2}})=\C \bar{D}_{1}\bigoplus \C \bar{D}_{2}\bigoplus\C \bar{D}_{3}.$$
\\
{\bf Subcase 2.2:} $b=1$.  Then $a_{12}^0=0$. By (\ref{q51}) and (\ref{q53}), we have
\begin{equation}\label{44}
(m-n)a_{12}^{m+n}=(m+n)[a_{12}^m-a_{12}^n],\quad\quad\quad\; m,n\in\Z.
\end{equation}
\begin{equation}\label{q71}
a_{22}^{m+n}=a_{11}^m+a_{22}^n,\quad\quad\quad\; m+n\neq 0.
\end{equation}
Letting $n=\pm 1$ respectively in (\ref{44}) and using induction on $m$,   we can deduce
$$a_{12}^m=\dis\frac{m(m-1)}{2}a_{12}^{-1}+\dis\frac{m(m+1)}{2}a_{12}^1=\dis\frac{m(m-1)}{2}[a_{12}^{-1}+a_{12}^1]+ma_{12}^1,$$
for all $m\in\Z$.
It is easy following (\ref{q71}) to see that$$a_{22}^{m}=a_{11}^m+a_{22}^0=ma_{11}^1+a_{22}^0, \ \; m\neq 0.$$
So we get $$a_{22}^{m}=ma_{11}^1+a_{22}^0, \ \ \; \forall m\in\Z.$$
Letting $n=0$ in (\ref{q54}), we get $$b^{m+\frac{1}{2}}=ma_{11}^{1}+b^{\frac{1}{2}}=(m+\frac{1}{2})a_{11}^{1}+\frac{1}{2}a_{22}^{0},$$ for all $m\in\Z.$
Therefore,
$$D(L_{m})=ma_{11}^1L_m+[\dis\frac{m(m-1)}{2}a_{12}^{-1}+\dis\frac{m(m+1)}{2}a_{12}^1]I_{m},$$
$$D(I_{m})=(ma_{11}^1+a_{22}^0)I_{m}, \ \; D(Y_{m+\frac{1}{2}})=[(m+\frac{1}{2})a_{11}^{1}+\frac{1}{2}a_{22}^{0}]Y_{m+\frac{1}{2}}.$$
Set $E_{0}=-a_{11}^1L_0+a_{12}^1I_{0}$, then
$$D(L_{m})=adE_{0}(L_{m})+\dis\frac{m(m-1)}{2}[a_{12}^{-1}+a_{12}^1]I_{m},$$
$$ D(I_{m})=adE_{0}(I_{m})+a_{22}^0I_{m},\ \; D(Y_{m+\frac{1}{2}})=adE_0(Y_{m+\frac{1}{2}})+\frac{1}{2}a_{22}^{0}Y_{m+\frac{1}{2}}.$$
 Let $\bar{D}_{1}, \bar{D}_{2}\in {\rm
H}^{1}(\W^{\rm\bf g}(0,1)_{\frac{m}{2}}, \W^{\rm\bf g}(0,1)_{\frac{m}{2}})$ such that
$$\bar{D}_{1}(L_{m})=0, \ \ \ \ \ \ \bar{D}_{1}(I_{m})=I_{m},\ \ \ \ \bar{D}_1(Y_{m+\frac{1}{2}})=Y_{m+\frac{1}{2}};$$
$$\bar{D}_{2}(L_{m})=m(m-1)I_{m}, \ \ \ \  \bar{D}_{2}(I_{m})=0, \ \ \ \bar{D}_2(Y_{m+\frac{1}{2}})=0,$$
for all $m\in\Z$, then we have
$${\rm H}^{1}(\W^{\rm\bf g}(0,1)_{\frac{m}{2}}, \W^{\rm\bf g}(0,1)_{\frac{m}{2}})=\C\bar{D}_{1}\bigoplus\C\bar{D}_{2}.$$

{\bf Subcase 2.3:} $b\neq 0,1$.  Then we always have
 $a_{12}^{-m}=-a_{12}^m$ from (\ref{39}) for all $m\in\Z$.
Let $n=1$ in (\ref{q51}), we have
$$(m-1)a_{12}^{m+1}=(b+m)a_{12}^m-(bm+1)a_{12}^1.$$
Then\begin{equation}\label{46}
(m-1)[a_{12}^{m+1}-(m+1)a_{12}^1]=(m+b)[a_{12}^m-ma_{12}^1].\end{equation}
Let $m=-2$ in (\ref{46}), we have
$$(b-2)[a_{12}^2-2a_{12}^1]=0.$$
Hence
 $$a_{12}^2=2a_{12}^1,\ \ \ b\neq 2.$$
By induction on $m$ in (\ref{46}),  we can deduce that
$$a_{12}^m=ma_{12}^1, \ \ \ \forall m\in\Z.$$
By (\ref{38}), we have $$a_{22}^m=a_{11}^m+a_{22}^0=ma_{11}^1+a_{22}^0, \ \ \; m\in\Z.$$
Similar to the computations in subcase1.1, we have
$$b^{m+\frac{1}{2}}=ma_{11}^{1}+b^{\frac{1}{2}}=(m+\frac{1}{2})a_{11}^{1}+\frac{1}{2}a_{22}^{0}, \ \ \forall m\in\Z.$$
So
$$D(L_{m})=ma_{11}^1L_m+ma_{12}^1I_{m}, \ \ D(I_{m})=(ma_{11}^1+a_{22}^0)I_{m},$$
$$D(Y_{m+\frac{1}{2}})=[(m+\frac{1}{2})a_{11}^{1}+\frac{1}{2}a_{22}^{0}]Y_{m+\frac{1}{2}}.$$
Set $E_{0}=-a_{11}^1L_0+\dis\frac{a_{12}^1}{b}I_{0}$, then
$$D(L_{m})=ad E_{0}(L_{m}), \ \ D(I_{m})=ad E_{0}(I_{m})+a_{22}^{0}I_{m},$$
$$D(Y_{m+\frac{1}{2}})=adE_0(Y_{m+\frac{1}{2}})+\frac{1}{2}a_{22}^{0}Y_{m+\frac{1}{2}}.$$
Consequently, for $b\neq 0, 1, 2$, we have
$${\rm H}^{1}(\W^{\rm\bf g}(0,b)_{\frac{m}{2}}, \W^{\rm\bf g}(0,b)_{\frac{m}{2}})=\C\bar{D},$$
where $\bar{D}(L_{m})=0, \ \ \bar{D}(I_{m})=I_{m}, \ \ \bar{D}(Y_{m+\frac{1}{2}})=Y_{m+\frac{1}{2}}$ for all $m\in\Z$.

If $b=2$, by induction on $m$ in (\ref{46}), we have
$$a_{12}^m=\frac{m^{3}-m}{6}a_{12}^2-\frac{m^{3}-4m}{3}a_{12}^1=\frac{a_{12}^2-2a_{12}^1}{6}m^{3}-\frac{a_{12}^2-8a_{12}^1}{6}m,$$
for all $m\in\Z$. Then
$$D(L_{m})=ma_{11}^1L_m+(\frac{a_{12}^2-2a_{12}^1}{6}m^{3}-\frac{a_{12}^2-8a_{12}^1}{6}m)I_{m}, \ \ D(I_{m})=(ma_{11}^1+a_{22}^0)I_{m},$$
$$D(Y_{m+\frac{1}{2}})=[(m+\frac{1}{2})a_{11}^{1}+\frac{1}{2}a_{22}^{0}]Y_{m+\frac{1}{2}}.$$
Set $E_{0}=-a_{11}^1L_0-\dis\frac{a_{12}^2-8a_{12}^1}{12}I_{0}$, then
$$D(L_{m})=ad E_{0}(L_{m})+\frac{a_{12}^2-2a_{12}^1}{6}m^{3}I_{m}, \ \ D(I_{m})=ad E_{0}(I_{m})+a_{22}^{0}I_{m},$$
$$D(Y_{m+\frac{1}{2}})=adE_0(Y_{m+\frac{1}{2}})+\frac{1}{2}a_{22}^{0}Y_{m+\frac{1}{2}}.$$
Consequently, we have
$${\rm H}^{1}(\W^{\rm\bf g}(0,2)_{\frac{m}{2}}, \W^{\rm\bf g}(0,2)_{\frac{m}{2}})=\C\bar{D}_{1}\bigoplus\C\bar{D}_{2},$$
where $\bar{D}_{1}(L_{m})=0, \ \ \bar{D}_{1}(I_{m})=I_{m}, \ \; \bar{D}_{1}(Y_{m+\frac{1}{2}})=Y_{m+\frac{1}{2}}; \ \
\bar{D}_{2}(L_{m})=m^{3}I_{m}, \ \ \bar{D}_{2}(I_{m})=0, \ \ \\ \bar{D}_{2}(Y_{m+\frac{1}{2}})=0$ for all
$m\in\Z$.

{\bf Case 3:} $a=1$. Similar to the above discussions for the case $a\not\in\Z$ completely, we only need to take $a=1$ in these discussions
and get the same results as case 1.
\end{proof}
\begin{lem}Up to isomorphism,
$${\rm H}^{1}(\Wgab_{\frac{m}{2}}, \Wgab_{\frac{m}{2}\pm 1})=0.$$\end{lem}
\begin{proof} By \cite{SCY}, $\forall \bar{D}\in {\rm H}^{1}(\Wgab_{\frac{m}{2}}, \Wgab_{\frac{m}{2}+1}),$ we may assume that$$\bar{D}(L_m)=0, \ \ \bar{D}(I_m)=0, \ \ \bar{D}(Y_{m+\frac{1}{2}})=b^{m+1+\frac{1}{2}}Y_{m+1+\frac{1}{2}}.$$
By $\bar{D}([L_m, Y_{n+\frac{1}{2}}])=[\bar{D}(L_m), Y_{n+\frac{1}{2}}]+[L_m, \bar{D}(Y_{n+\frac{1}{2}})]$, we have
\begin{equation}\label{q72}(n+\dis\frac{1-m+a+bm}{2})b^{m+n+1+\frac{1}{2}}=(n+1+\dis\frac{1-m+a+bm}{2})b^{n+1+\frac{1}{2}}
\end{equation}
Let $m=0$ in (\ref{q72}), we get $b^{n+1+\frac{1}{2}}=0$ for all $n\in\Z$. So $\bar{D}(Y_{m+\frac{1}{2}})=0$. Consequently, we get
$${\rm H}^{1}(\Wgab_{\frac{m}{2}}, \Wgab_{\frac{m}{2}+1})=0.$$
For ${\rm H}^{1}(\Wgab_{\frac{m}{2}}, \Wgab_{\frac{m}{2}-1})$, similar to the above discussion, we have the same results.
\end{proof}
\begin{lem}Up to isomorphism,
$${\rm H}^{1}(\Wgab_{\frac{m}{2}}, \Wgab_{\frac{m+1}{2}})=0.$$\end{lem}
\begin{proof}
$\forall D\in\Der(\Wgab)_{\frac{1}{2}},$ we assume that
$$D(L_m)=b_1^{m+\frac{1}{2}}Y_{m+\frac{1}{2}},\  D(I_m)=b_2^{m+\frac{1}{2}}Y_{m+\frac{1}{2}},\  D(Y_{m+\frac{1}{2}})=a_{11}^{m+1}L_{m+1}+a_{12}^{m+1}I_{m+1}.$$ By the definition of derivation and the bracket for $\Wgab$,
we have\begin{equation}\label{q73}
 (m-n)b_1^{m+n+\frac{1}{2}}=(m+\dis\frac{1-n+a+bn}{2})b_1^{m+\frac{1}{2}}-(n+\dis\frac{1-m+a+bm}{2})b_1^{n+\frac{1}{2}}\end{equation}
\begin{equation}\label{q74}(n+a+bm)b_2^{m+n+\frac{1}{2}}=(n+\dis\frac{1-m+a+bm}{2})b_2^{n+\frac{1}{2}}\end{equation}
\begin{equation}\label{q75}(n+\dis\frac{1-m+a+bm}{2})a_{11}^{m+n+1}=(n+1-m)a_{11}^{n+1}\end{equation}
\begin{equation}\label{q76}(n+\dis\frac{1-m+a+bm}{2})a_{12}^{m+n+1}=(n-m)b_1^{m+\frac{1}{2}}+(n+1+a+bm)a_{12}^{n+1}\end{equation}
\begin{equation}\label{q77}(m-n)b_2^{m+\frac{1}{2}}+(m+a+b(n+1))a_{11}^{n+1}=0\end{equation}
\begin{equation}\label{q78}(m-n)b_2^{m+n+1+\frac{1}{2}}=-(n+\dis\frac{-m+a+b(m+1)}{2})a_{11}^{m+1}+(m+\dis\frac{-n+a+b(n+1)}{2})a_{11}^{n+1}
\end{equation}
Let $n=0$ in (\ref{q73}), we have $(1+a)b_1^{m+\frac{1}{2}}=(1-m+a+bm)b_1^{\frac{1}{2}}$. Since $a\not\in\Z$ or $a=0, a=1$, we get
$$b_1^{m+\frac{1}{2}}=\dis\frac{1+a-m+bm}{1+a}b_1^{\frac{1}{2}}, \ \ \ \forall m\in\Z.$$
Let $m=0$ in (\ref{q74}), (\ref{q75}) and (\ref{q76}), we have
\begin{equation}\label{q79}(a-1)b_2^{n+\frac{1}{2}}=0,\end{equation}
\begin{equation}\label{q80}(a-1)a_{11}^{n+1}=0,\end{equation}
\begin{equation}\label{q81}(a+1)a_{12}^{n+1}=-2nb_1^{\frac{1}{2}}.\end{equation}
Then we get $a_{12}^{n+1}=\dis\frac{-2n}{a+1}b_1^{\frac{1}{2}},$ for all $n\in\Z$ since $a\not\in\Z$ or $a=0, a=1$.
Let $m=-1$ in (\ref{q77}) and use (\ref{q80}), we obtain \begin{equation}\label{q82}(n+1)ba_{11}^{n+1}=(n+1)b_2^{-1+\frac{1}{2}}, \ \ \forall n\in\Z.\end{equation}
{\bf Case 1:} $a\not\in\Z$ or $a=0$. It follows from (\ref{q79}) and (\ref{q80}) that $$b_2^{n+\frac{1}{2}}=0, \ \ \ a_{11}^{n+1}=0, \ \ \forall n\in\Z.$$ So $$D(L_m)=\dis\frac{1+a-m+bm}{1+a}b_1^{\frac{1}{2}}Y_{m+\frac{1}{2}},\  \ D(I_m)=0,\  D(Y_{m+\frac{1}{2}})=\dis\frac{-2m}{1+a}b_1^{\frac{1}{2}}I_{m+1}.$$
Set $E=\dis\frac{2b^{\frac{1}{2}}}{1+a}Y_{\frac{1}{2}}$, then
$$D(L_m)=adE(L_m), \ \ D(I_m)=adE(I_m), \ \ D(Y_{m+\frac{1}{2}})=adE(Y_{m+\frac{1}{2}}).$$
{\bf Case 2:} $a=1$. Let $n=-1$ in (\ref{q77}), we have \begin{equation}\label{q83} b_2^{m+\frac{1}{2}}=-a_{11}^{0}, \ \ \ m\neq -1.\end{equation}
By (\ref{q82}), we also have \begin{equation}\label{q84} b_2^{-1+\frac{1}{2}}=ba_{11}^{n+1}, \ \ \ \ n\neq -1.\end{equation}
Let $n=0, m=1$ in (\ref{q74}) and use (\ref{q83}), we get \begin{equation}\label{q85}(1+b)a_{11}^{0}=0.\end{equation}
On the other hand, let $m+n=-1$ in (\ref{q75}), we get
\begin{equation}\label{q86}a_{11}^{n+1}=\dis\frac{3-b}{4}a_{11}^{0},\ \ \ \ n\neq -1.\end{equation}
{\bf Subcase 2.1:} $b=0$. From (\ref{q83})-(\ref{q86}), we get $$b_2^{m+\frac{1}{2}}=0, \ \ \ a_{11}^{n+1}=0, \ \ \ \forall m, n\in\Z.$$
So we have the same results as case1.\\
{\bf Subcase 2.2:} $b\neq 0$. If $b\neq -1$. By (\ref{q85}), we have $a_{11}^{0}=0,$ then by (\ref{q83}), (\ref{q84}), (\ref{q86}), we obtain
 $$b_2^{m+\frac{1}{2}}=0, \ \ \ a_{11}^{n+1}=0, \ \ \ \forall m, n\in\Z.$$ So we get the same results as subcase 2.1 completely.\\
If $b=-1$, By (\ref{q83}), (\ref{q84}), (\ref{q86}), we have
\begin{equation}\label{L8}a_{11}^{n+1}=a_{11}^{0}, \ \ \ b_2^{m+\frac{1}{2}}=-a_{11}^{0}, \ \ \forall m, n\in\Z.\end{equation}
On the other hand, let $n=0, m=1$ in (\ref{q78}) and use (\ref{L8}), we obtain $a_{11}^{0}=0$. Then $a_{11}^{n+1}=b_2^{n+\frac{1}{2}}=0$, for all $n\in\Z$ and we get the same results as subcase 2.1.
\end{proof}
\begin{lem}\label{L3.8}Up to isomorphism,
\begin{eqnarray}\nonumber
{\rm H}^{1}(\Wgab_{\frac{m}{2}}, \Wgab_{\frac{m-1}{2}})=\left\{
\begin{array}{ll}
&\C\bar{D}, \ \ \ (a, b)=(1, -1) \vs{4pt}\\
&0, \ \ \ \ \ \ \ else
\end{array}
\right.
\end{eqnarray}
where$$\bar{D}(L_m)=0, \ \bar{D}(I_m)=0, \ \bar{D}(Y_{m+\frac{1}{2}})=I_m.$$\end{lem}
\begin{proof}$\forall D\in\Der(\Wgab)_{-\frac{1}{2}},$ we assume that
$$D(L_m)=b_1^{m-\frac{1}{2}}Y_{m-1+\frac{1}{2}},\  D(I_m)=b_2^{m-\frac{1}{2}}Y_{m-1+\frac{1}{2}},\  D(Y_{m+\frac{1}{2}})=a_{11}^{m}L_{m}+a_{12}^{m}I_{m}.$$ By the definition of derivation and the bracket for $\Wgab$,
we have\begin{equation}\label{q87}
 (m-n)b_1^{m+n-\frac{1}{2}}=(m-1+\dis\frac{1-n+a+bn}{2})b_1^{m-\frac{1}{2}}-(n-1+\dis\frac{1-m+a+bm}{2})b_1^{n-\frac{1}{2}}\end{equation}
\begin{equation}\label{q88}(n+a+bm)b_2^{m+n-\frac{1}{2}}=(n-1+\dis\frac{1-m+a+bm}{2})b_2^{n-\frac{1}{2}}\end{equation}
\begin{equation}\label{q89}(n+\dis\frac{1-m+a+bm}{2})a_{11}^{m+n}=(n-m)a_{11}^{n}\end{equation}
\begin{equation}\label{q90}(n+\dis\frac{1-m+a+bm}{2})a_{12}^{m+n}=(n+1-m)b_1^{m-\frac{1}{2}}+(n+a+bm)a_{12}^{n}\end{equation}
\begin{equation}\label{q91}(m-1-n)b_2^{m-\frac{1}{2}}+(m+a+bn)a_{11}^{n}=0\end{equation}
\begin{equation}\label{q92}(m-n)b_2^{m+n+\frac{1}{2}}=-(n+\dis\frac{1-m+a+bm}{2})a_{11}^{m}+(m+\dis\frac{1-n+a+bn}{2})a_{11}^{n}
\end{equation}
Let $n=0$ in (\ref{q87}) and $m=0$ in (\ref{q88})-(\ref{q90}), we have the following identities
\begin{equation}\label{q93}(a-1)b_1^{m-\frac{1}{2}}=(-1-m+a+bm)b_1^{-\frac{1}{2}},\end{equation}
\begin{equation}\label{q94}(a-1)a_{12}^{n}=-2(n+1)b_1^{-\frac{1}{2}},\end{equation}
$$(a+1)b_2^{n-\frac{1}{2}}=0, \ \ \ \ (a+1)a_{11}^{n}=0.$$ Since $a\not\in\Z$ or $a=0$ or $a=1$, we get
$$b_2^{n-\frac{1}{2}}=0, \ \ \ a_{11}^{n}=0, \ \ \forall n\in\Z.$$
{\bf Case 1:} $a\not\in\Z$ or $a=0$. By (\ref{q93}) and (\ref{q94}), we get
$$b_1^{m-\frac{1}{2}}=\dis\frac{-1+a-m+bm}{a-1}b_1^{-\frac{1}{2}}, \ \ a_{12}^{n}=\dis\frac{-2(n+1)}{a-1}b_1^{-\frac{1}{2}}, \ \ \forall m, n\in\Z.$$
Then $$D(L_m)=\dis\frac{-1+a-m+bm}{a-1}b_1^{-\frac{1}{2}}Y_{m-\frac{1}{2}}, \ \ D(I_m)=0, \ \ D(Y_{m+\frac{1}{2}})=\dis\frac{-2(m+1)}{a-1}b_1^{-\frac{1}{2}}I_m.$$
Set $E=\dis\frac{2b_1^{-\frac{1}{2}}}{a-1}Y_{-\frac{1}{2}}$, so
$$D(L_m)=adE(L_m), \ \ D(I_m)=adE(I_m), \ \ D(Y_{m+\frac{1}{2}})=adE(Y_{m+\frac{1}{2}}).$$
{\bf Case 2:} $a=1$. By (\ref{q94}), we have $b_1^{-\frac{1}{2}}=0$. So let $m+n=0$ in (\ref{q87}), we have
$$(3-b)(b_1^{m-\frac{1}{2}}+b_1^{-m-\frac{1}{2}})=0,$$ then $$(b_1^{m-\frac{1}{2}}+b_1^{-m-\frac{1}{2}})=0, \ \ b\neq 3, \forall m\in\Z.$$
On the other hand, let $n=1$ in (\ref{q87}), we have
$$(m-1)b_1^{m+1-\frac{1}{2}}=(m+\frac{b-1}{2})b_1^{m-\frac{1}{2}}-(1+\frac{b-1}{2}m)b_1^{1-\frac{1}{2}},$$ that is
\begin{equation}\label{q95}(m-1)(b_1^{m+1-\frac{1}{2}}-(m+1)b_1^{1-\frac{1}{2}})=(m+\frac{b-1}{2})(b_1^{m-\frac{1}{2}}-mb_1^{1-\frac{1}{2}})\end{equation}
Induction on $m\in\Z, m\geq 3$ in (\ref{q95}), we get \begin{equation}\label{q103}b_1^{m-\frac{1}{2}}=\dis\frac{(3+b)(5+b)\cdots(2m-3+b)}{2^{m-2}(m-2)!}(b_1^{2-\frac{1}{2}}-2b_1^{1-\frac{1}{2}})
+mb_1^{1-\frac{1}{2}}.\end{equation}
Let $n=-1$ in (\ref{q90}), we have \begin{equation}\label{q96}\frac{b-1}{2}ma_{12}^{m-1}=bma_{12}^{-1}-mb_1^{m-\frac{1}{2}}.\end{equation}
{\bf Subcase 2.1:} $b=1$. By (\ref{q95}) and using induction, we have
\begin{equation}\label{q97}b_1^{m-\frac{1}{2}}=(m-1)b_1^{2-\frac{1}{2}}-(m-2)b_1^{1-\frac{1}{2}}, \ \ \forall m\in\Z, m\neq 0.\end{equation}
On the other hand, by (\ref{q96}), we get
\begin{equation}\label{q98}b_1^{m-\frac{1}{2}}=a_{12}^{-1}, \ \forall m\in\Z,  m\neq 0.\end{equation}
Let $n=0$ in (\ref{q90}) and use (\ref{q98}), we obtain
\begin{equation}\label{q99}a_{12}^{m}=(m+1)a_{12}^{0}-(m-1)a_{12}^{-1}, \ m\neq 0.\end{equation}
Let $m=-1$ in (\ref{q99}), we get $a_{12}^{-1}=0$. Consequently, we obtain
$$b_1^{m-\frac{1}{2}}=0, \ \ \ a_{12}^{m}=(m+1)a_{12}^{0}, \ \ \forall m\in\Z.$$
 So
$$D(L_m)=0, \ \ D(I_m)=0, \ \ D(Y_{m+\frac{1}{2}})=(m+1)a_{12}^{0}I_m.$$
Set $E=-a_{12}^{0}Y_{-\frac{1}{2}}$, then we have$$
D(L_m)=adE(L_m), \ D(I_m)=adE(I_m), \ D(Y_{m+\frac{1}{2}})=adE(Y_{m+\frac{1}{2}}).$$
{\bf Subcase 2.2:} $b\neq 1$. By (\ref{q96}), we have \begin{equation}\label{q101}a_{12}^{m-1}=\frac{2b}{b-1}a_{12}^{-1}-\frac{2}{b-1}b_1^{m-\frac{1}{2}}, \ \ m\neq 0.\end{equation}
Let $n+1-m=0$ in (\ref{q90}), we have $$(n+1)(b+1)a_{12}^{2n+1}=2(n+1)(b+1)a_{12}^{n},$$ then
\begin{equation}\label{q100}(b+1)a_{12}^{2n+1}=2(b+1)a_{12}^{n}, \ \ n\neq -1.\end{equation}
If $b\neq -1$, we have $a_{12}^{2n+1}=2a_{12}^{n}, \ n\neq -1.$ using (\ref{q101}), we get\begin{equation}\label{q102}b_1^{2m+2-\frac{1}{2}}=2b_1^{m+1-\frac{1}{2}}-ba_{12}^{-1}, \ m\neq -1.\end{equation}
Especially, $b_1^{2-\frac{1}{2}}-2b_1^{1-\frac{1}{2}}=-ba_{12}^{-1}$. Combine (\ref{q102}) and (\ref{q103}), we get
\begin{equation}\label{q104}ba_{12}^{-1}=0.\end{equation}
Consequently, we get
$$b_1^{m-\frac{1}{2}}=mb_1^{1-\frac{1}{2}}, \ \ a_{12}^{m-1}=-\frac{2m}{b-1}b_1^{1-\frac{1}{2}}, \ \ \forall m\in\Z.$$
Set $E=\dis\frac{2}{b-1}b_1^{1-\frac{1}{2}}Y_{-\frac{1}{2}}$, then we have$$
D(L_m)=adE(L_m), \ D(I_m)=adE(I_m), \ D(Y_{m+\frac{1}{2}})=adE(Y_{m+\frac{1}{2}}).$$
If $b=-1$, by (\ref{q103}), we still have
$$b_1^{m-\frac{1}{2}}=mb_1^{1-\frac{1}{2}}, \ \ a_{12}^{m-1}=a_{12}^{-1}+mb_1^{1-\frac{1}{2}},\, \forall m\in\Z.$$
Set $E=-b_1^{1-\frac{1}{2}}Y_{-\frac{1}{2}}$, then we have$$
D(L_m)=adE(L_m), \ D(I_m)=adE(I_m), \ D(Y_{m+\frac{1}{2}})=adE(Y_{m+\frac{1}{2}})+a_{12}^{-1}I_m.$$
Then we get $${\rm H}^{1}(\Wgab_{\frac{m}{2}}, \Wgab_{\frac{m-1}{2}})=\C\bar{D},$$
Where $\bar{D}(L_m)=0, \bar{D}(I_m)=0, \bar{D}(Y_{m+\frac{1}{2}})=I_m.$
\end{proof}
Now, by lemma \ref{L3.6}-\ref{L3.8}, we obtain the main theorem of this section as following.
\begin{theorem}\label{T3} Up to isomorphism, we have
\begin{eqnarray}\nonumber
{\rm H}^{1}(\Wgab, \Wgab)=\left\{
\begin{array}{ll}
&\C D_{1}\bigoplus \C D_{2}^{0,0}\bigoplus \C D_{3}, \ \ \ \; (a,
b)=(0, 0),
\vs{4pt}\\
&\C D_{1}\bigoplus \C D_{2}^{0,1}, \ \ \quad\quad\quad\quad \; (a, b)=(0, 1),
\vs{4pt}\\
&\C D_{1}\bigoplus \C D_{2}^{0,2}, \ \ \ \quad\quad\quad\quad  (a, b)=(0,2),
\vs{4pt}\\
&\C D_{3}^{1,-1}, \ \ \ \quad\quad\quad\quad\quad\quad\quad (a, b)=(1, -1),
\vs{4pt}\\
&\C D_{1}, \ \ \ \quad\quad\quad\quad\quad\quad\quad\quad otherwise,
\end{array}
\right.
\end{eqnarray}
where for all $m\in\Z$,
\begin{eqnarray}
&&D_{1}(L_{m})=0, \quad\quad\quad\quad\quad\quad\quad D_{1}(I_{m})=I_{m}, \quad\quad\quad D_{1}(Y_{m+\frac{1}{2}})=Y_{m+\frac{1}{2}},
\vs{4pt}\\
&&D_{2}^{0,0}(L_{m})=(m-1)I_{m}, \quad\quad\quad D_{2}^{0,0}(I_{m})=0,\quad\quad\quad D_{2}^{0,0}(Y_{m+\frac{1}{2}})=0,
\vs{4pt}\\
&&D_{2}^{0,1}(L_{m})=(m^2-m)I_{m}, \quad\quad\; D_{2}^{0,1}(I_{m})=0,\quad\quad\quad D_{2}^{0,1}(Y_{m+\frac{1}{2}})=0,
\vs{4pt}\\
&&D_{2}^{0,2}(L_{m})=m^3 I_m, \quad\quad\quad\quad\quad D_{2}^{0,2}(I_{m})=0,\quad\quad\quad D_{2}^{0,2}(Y_{m+\frac{1}{2}})=0,
\vs{4pt}\\
&&D_3(L_m)=mI_m, \quad\quad\quad\quad\quad\quad D_3(I_m)=0,\quad\quad\quad\quad D_{3}(Y_{m+\frac{1}{2}})=0,
\vs{4pt}\\
&&D_3^{1,-1}(L_m)=0, \quad\quad\quad\quad\quad\quad D_3^{1,-1}(I_m)=0,\quad\quad\quad D_{3}^{1,-1}(Y_{m+\frac{1}{2}})=I_m.
\end{eqnarray}
\end{theorem}

%%%%%%%%%%%%%%%%%%%%%%%%%%%%%%%%%%%%%%%%%%%%%%%%%%%%%%%%%%%%%%%%%%%%%%%%%%%%%%%%%%%%%%%%%%%%%%%%%%%%%%
%
\section{\bf Automorphism groups of $\Wgab$ }
%
%%%%%%%%%%%%%%%%%%%%%%%%%%%%%%%%%%%%%%%%%%%%%%%%%%%%%%%%%%%%%%%%%%%%%%%%%%%%%%%%%%%%%%%%%%%%%%%%%%%%
\label{sub7-c-related}

Denote by  $\mathfrak{I}$  the inner automorphism group of ${\Wgab}$.
Obviously,  $\mathfrak{I}$  is generated by $\{{\rm exp\,}k_{i}\, {\rm ad}\,I_{i}, {\rm
exp\,}l_{i}\, {\rm ad}\,Y_{i}\mid k_i,l_i\in\C,\,i,j\in \Z\}$ and  is a normal subgroup of
$\Aut({\Wgab})$.  We can verify easily that $\mathfrak{I}$ is isomorphism to $\C^{\infty}\times \C^{\infty}$ as sets. But by computation, these generators satisfy the following relation: $$ ({\rm exp}\, \alpha\, {\rm ad\,}Y_{j+\frac{1}{2}})({\rm
exp}\, \beta\,{\rm ad}\,Y_{i+\frac{1}{2}})({\rm exp} \,\alpha\,{\rm
ad\,}Y_{j+\frac{1}{2}})^{-1} ({\rm exp}\, \beta\, {\rm ad\,}Y_{i+\frac{1}{2}})^{-1}={\rm exp}\,\gamma\, {\rm ad\,}I_{i+j+1},$$
 where $\gamma=\alpha\beta(j-i)$. And the others is commutable each other. Hence $\mathfrak{I}$ is isomorphism to $\C^{\infty}\rtimes \C^{\infty}$ as groups.

 For any $\prod\limits_{j=s}^{t}\exp(k_{i_{j}}\ad I_{i_{j}}+l_{i_{j}}\ad Y_{i_{j}+\frac{1}{2}})\in
\mathfrak{I}$, we have
\begin{equation}\label{E4-0}
\prod\limits_{j=s}^{t}\exp(k_{i_{j}}\ad I_{i_{j}}+l_{i_{j}}\ad Y_{i_{j}+\frac{1}{2}})(I_{n})=I_{n},
\end{equation}
\begin{equation}\label{E4-f1}
\prod\limits_{j=s}^{t}\exp(k_{i_{j}}\ad I_{i_{j}}+l_{i_{j}}\ad Y_{i_{j}+\frac{1}{2}})
(Y_{n+\frac{1}{2}})=Y_{n+\frac{1}{2}}+\sum\limits_{j=s}^{t}l_{i_{j}}(i_{j}-n)I_{i_{j}+n+1}.
\end{equation}
As $\Igab$ is an unique maximal proper ideal of $\Wgab$, we
have the following lemma.
\begin{lem}\label{L4.1}
For any $\si\in \Aut({\Wgab})$, $\si(I_{n}), \ \si(Y_{n+\frac{1}{2}})\in\Igab$ for all
$n\in\Z$.
\end{lem}\qed\\
For any $\si\in \Aut({\Wgab})$, denote $\si|_{\W}=\si'$. Then
according to  the automorphisms of the classical Witt algebra, we
have $\si'(L_{m})=\epsilon \a^{m}L_{\epsilon m}$ for all $m\in\Z$,
where $\a\in\C^{*}$ and $\epsi\in\{\pm1\}$. So we have the following lemmas.

\begin{lem}\label{L4.f1}
For $a\not\in\Z$, $\Aut({\Wgab})\cong(\C^{\infty}\rtimes\C^{\infty})\rtimes(\C^{*}\times \C^{*})$.\\
where  $\C^{\infty}=\{(a_{i})_{i\in\Z}\;|\; a_{i}\in\C, {\rm{\; all
\; but\;  \; a \;  finite \; number \; of \;  the}} \; a_{i} \;
{\rm{ \; are \; zero}} \ \}$, $\C^{*}=\C\backslash\{0\}$.
\end{lem}

\begin{proof}
For any $\si\in \Aut({\Wgab})$,
assume
$$\si(L_{0})= \epsi L_{0}+\sum\limits_{i=p }^{q} \l_{i} I_{i}+\sum\limits_{j=s }^{t} l_{j} Y_{j+\frac{1}{2}},$$
where $\l_{i}, \ l_j\in\C$, $p, q, s, t\in\Z$,  $p\leq i\leq q, \ s\leq j\leq t$.\\
Since $a\not\in \Z$, we let
$\tau_{1}=\prod\limits_{i=p}^{q}\exp(\frac{\l_{i}}{\epsi (i+a)}\ad
I_{i})\prod\limits_{j=s}^{t}\exp(\frac{l_{j}}{\epsi (j+\frac{1+a}{2})}\ad Y_{j+\frac{1}{2}})\in \mathfrak{I}$,
then
$$\tau_1(\epsi L_{0})=\epsi L_0+\sum\limits_{i=p }^{q} \l_{i} I_{i}+\sum\limits_{j=s }^{t} l_{j} Y_{j+\frac{1}{2}}+
\sum\limits_{k=p'}^{q'} b_{k} I_{k}=\si(L_{0})+\sum\limits_{k=p'}^{q'} b_{k} I_{k},$$
where $b_k\in\C, \ p', q'\in\Z$. On the other hand, let $\tau_2=\prod\limits_{k=p'}^{q'}\exp(\frac{-b_{k}}{\epsi (k+a)}\ad
I_{k})\in \mathfrak{I}$ and $\tau=\tau_2\tau_1\in \mathfrak{I}$, we have $\tau(\epsi L_0)=\si(L_0)$.
Consequently, $\bar{\si}=\tau^{-1}\si$ and $\bar{\si}(L_{0})=\epsi L_{0}.$\\
By Lemma \ref{L4.1}, we may assume
$$\bar{\si}(L_{n})= \a^{n}\epsi L_{\epsi
n}+\sum\l_{n_{i}}I_{n_{i}}+\sum\mu_{m_j}Y_{m_j+\frac{1}{2}}, \;\; n\neq 0,$$
$$\bar{\si}(I_{m})=\sum c_{p}I_{p}+\sum d_{q}Y_{q+\frac{1}{2}}, $$
$$\bar{\si}(Y_{m+\frac{1}{2}})=\sum e_{u}I_{u}+\sum f_{v}Y_{v+\frac{1}{2}}, $$
where each formula
is of finite terms and $\l_{n_{i}}, \mu_{m_{j}}, c_{p}, d_{q}, e_{u}, f_{v}\in\C, n_{i}, m_{j}, p,
q, u, v\in\Z$, $\a\in\C^*, \epsi=\pm 1$. For any $n\neq 0$, by  the relation
$[\bar{\si}(L_{0}),\bar{\si}(L_{n})]=-n\bar{\si}(L_{n})$, we have
$$\l_{n_{i}}[\epsi(n_{i}+a)-n]=0, \ \mu_{m_j}[\epsi(m_{j}+\frac{1+a}{2})-n]=0.$$
Since $a\not\in Z$, this forces that $\l_{n_{i}}=0, \ \mu_{m_j}=0$ for all $n_i, m_j\in\Z$. So we obtain
$$\bar{\si}(L_{n})= \a^{n}\epsi L_{\epsi n}, \ \forall n\in Z.$$
Since $[\bar{\si}(L_{0}),\bar{\si}(I_{m})]=-(m+a)\bar{\si}(I_{m}), [\bar{\si}(L_{0}),\bar{\si}(Y_{m+\frac{1}{2}})]=-(m+\frac{1+a}{2})\bar{\si}(Y_{m+\frac{1}{2}})$,
we have
$$c_{p}[\epsi(p+a)-(m+a)]=0, \ \ d_{q}[\epsi(q+\frac{1+a}{2})-(m+a)]=0,$$
$$e_{u}[\epsi(u+a)-(m+\frac{1+a}{2})]=0, \ \ f_{v}[\epsi(v+\frac{1+a}{2})-(m+\frac{1+a}{2})]=0.$$
Therefore, if $\epsi=1$, then $p=m, v=m$ and all $d_{q}=0, e_{u}=0$; If $\epsi=-1$, then $p=-m-2a$ and $q=-m-\dis\frac{1+3a}{2}$, which implies that
$a\in\dis\frac{\Z}{2}$ and $a\in\dis\frac{2\Z-1}{3}$. We also have all $f_{v}=0$ and $u=-m-\dis\frac{1+3a}{2}$ which implies that $a\in\dis\frac{2\Z-1}{3}$ if $\epsi=-1$.  Since $a\not\in\Z$, either $a\in\dis\frac{\Z}{2}$ or $a\in\dis\frac{2\Z-1}{3}$ holds. \\
If $a\in\dis\frac{\Z}{2}$, from the above discussion, it forces that $\epsi=1$, otherwise it must be that $\bar{\si}(Y_{m+\frac{1}{2}})=0$, which is contradiction.
If $a\in\dis\frac{2\Z-1}{3}$, we have
$$\bar{\si}(I_{m})=d_{-m-\frac{1+3a}{2}}Y_{-m-\frac{1+3a}{2}+\frac{1}{2}}, \quad \   \bar{\si}(Y_{m+\frac{1}{2}})=e_{-m-\frac{1+3a}{2}}I_{-m-\frac{1+3a}{2}},$$ where
$d_{-m-\frac{1+3a}{2}},e_{-m-\frac{1+3a}{2}}\in\C^{*}$.
But by $[\bar{\si}(Y_{m+\frac{1}{2}}),\bar{\si}(Y_{n+\frac{1}{2}})]=(m-n)\bar{\si}(I_{m+n+1})$, we have
$d_{-m-n-1-\frac{1+3a}{2}}=0$ which means $\bar{\si}(I_{m})=0$. This also is a contradiction. So we obtain
$$\bar{\si}(L_{m})=\a^mL_m, \ \ \ \bar{\si}(I_{m})=c_mI_m, \ \ \bar{\si}(Y_{m+\frac{1}{2}})=f_mY_{m+\frac{1}{2}}, \ \ \forall m\in\Z.$$
where $\a, c_m, f_m\in\C^{*}$.\\
On the other hand, from the relations that$$[\bar{\si}(L_{m}),\bar{\si}(I_{n})]=-(n+a+bm)\bar{\si}(I_{m+n}),$$
$$[\bar{\si}(L_{m}),\bar{\si}(Y_{n+\frac{1}{2}})]=-(n+\dis\frac{1-m+a+bm}{2})\bar{\si}(Y_{m+n+\frac{1}{2}}),$$
$$[\bar{\si}(Y_{m+\frac{1}{2}}),\bar{\si}(Y_{n+\frac{1}{2}})]=(m-n)\bar{\si}(I_{m+n+1}),$$
we have$$(n+a+bm)(C_{m+n}-C_n\a^m)=0, \ \ \a^{m}f_n=f_{m+n}, \ \ f_mf_n=C_{m+n+1},$$
for all $m,n\in\Z$. It is easy to deduce that
$$C_m=\a^m\mu, \ \ \ f_m=\a^{m}\sqrt{\a\mu}, \ \ \forall m\in\Z,$$where $\mu$ is a nonzero complex number.
Therefore,
\begin{equation}\label{E4-1}
\bar{\si}(L_{m})= \a^{m} L_{ m},  \ \  \bar{\si}(I_{m})=\a^{m}\mu
I_{m}, \ \ \ \bar{\si}(Y_{m+\frac{1}{2}})=\a^{m}\sqrt{\a\mu} Y_{m+\frac{1}{2}} \ \ \forall\;m\in\Z.
\end{equation}
Conversely, if $\bar{\si}$ is a linear operator on ${\Wgab}$ satisfying
(\ref{E4-1}) for some  $\a,\mu\in\C^{*}$ , then  it is easy to check
that $\bar{\si}\in \Aut({\Wgab})$.
Denote by $\bar{\si}(\a, \mu)$ the automorphism of $\Wgab$
satisfying (\ref{E4-1}), then
\begin{equation}\label{E4-1-1}
\bar{\si}(\a_{1}, \mu_{1}) \bar{\si}(\a_{2},\mu_{2})
=\bar{\si}(\a_{1}\a_{2}, \mu_{1}\mu_{2}),
\end{equation}
and $\bar{\si}(\a_{1}, \mu_{1})=\bar{\si}(\a_{2}, \mu_{2})$ if and
only if $\a_{1}=\a_{2}, \mu_{1}= \mu_{2}$. Let
$$\mathfrak{a}_{1}=\{\bar{\si}_{\a, \mu} \; |\; \a, \mu\in\C^{*} \}. $$
By (\ref{E4-1-1}),  $\mathfrak{a}_{1}\cong \C^{*}\times \C^{*}$ is a
subgroup of $\Aut({\Wgab})$. On the other
hand, similar to the proof in \cite{GJP}, \cite{GJP1} or \cite{SJ}, $\mathfrak{I}\cong\C^{\infty}\rtimes\C^{\infty}$. Consequently, we
have
$$\Aut(\Wgab)=\mathfrak{I}\rtimes\mathfrak{a}_{1}\cong \mathfrak{I}\rtimes(\C^{*}\times \C^{*}),$$
where $a\not\in\Z$.
\end{proof}

\begin{lem}\label{L4.f2}
For $a=0$, $\Aut({\Wgab})\cong(\C^{\infty}\rtimes\C^{\infty})\rtimes(\Z_2\ltimes(\C^{*}\times \C^{*}))$.\\
where  $\C^{\infty}=\{(a_{i})_{i\in\Z}\;|\; a_{i}\in\C, {\rm{\; all
\; but\;  \; a \;  finite \; number \; of \;  the}} \; a_{i} \;
{\rm{ \; are \; zero}} \ \}$, $\Z_{2}=\Z/2\Z$, $\C^{*}=\C\backslash\{0\}$.
\end{lem}

\begin{proof}
For any $\si\in\Aut(\Wgab)$, $\si(L_0)$ is the same as in lemma \ref{L4.f1}. Since $a=0$.  Let
$\tau_1=\prod\limits_{i=p}^{q}\exp(\dis\frac{\l_{i}}{\epsi i}\ad
I_{i})\prod\limits_{j=s}^{t}\exp(\dis\frac{l_{j}}{\epsi (j+\frac{1}{2})}\ad Y_{j+\frac{1}{2}})\in \mathfrak{I}$, where $i\neq 0$.
We use the same idea as in Lemma \ref{L4.f1}, then there exists $\tau\in\mathfrak{I}$ such that
$\si(L_{0})=\tau(\epsi L_{0})+\l_{0}I_{0}.$ Set
$\bar{\si}=\tau^{-1}\si$, then $\bar{\si}(L_{0})=\epsi
L_{0}+\l_{0}I_{0}.$  As the assumption in Lemma \ref{L4.f1}, for  any $n\neq
0, m\in\Z$, since $$[\bar{\si}(L_{0}),\bar{\si}(L_{n})]=-n\bar{\si}(L_{n}), \ \  [\bar{\si}(L_{0}),\bar{\si}(I_{m})]=-m\bar{\si}(I_{m})$$
we have
$$\mu_{m_j}[\epsi(m_j+\dis\frac{1}{2})-n]=0,\ \ \  d_{q}[\epsi(q+\dis\frac{1}{2})-n]=0.$$
This forces that $\mu_{m_j}=0, d_q=0$ for all $m_j, q\in\Z$. Therefore, $$\bar{\si}(\Wab)=\Wab.$$ On the other hand,
according to
$[\bar{\si}(L_{0}),\bar{\si}(Y_{m+\frac{1}{2}})]=-(m+\dis\frac{1}{2})\bar{\si}(Y_{m+\frac{1}{2}})$, we have
$$e_u(\epsi u-m-\dis\frac{1}{2})=0, \ \  f_v[\epsi (v+\dis\frac{1}{2})-(m+\dis\frac{1}{2})]=0.$$
So $e_u=0, v=\epsi (m+\dis\frac{1}{2})-\dis\frac{1}{2}$ and
$$\bar{\si}(Y_{m+\frac{1}{2}})=f_{\epsi (m+\frac{1}{2})-\frac{1}{2}}Y_{\epsi (m+\frac{1}{2})},$$
for all $m\in\Z$. Consequently, by \cite{SCY}, we have the following cases.
\\
{\bf Case 1} $b=0$. Assume
\begin{equation}\nonumber
\bar{\si}(L_{m})=\epsi \a^{m}L_{\epsi m}+\a^{m}(cm+d)I_{\epsi m}, \
\ \bar{\si}(I_{m})=\a^{m}\mu I_{\epsi m},
\end{equation}
\begin{equation}\nonumber
\bar{\si}(Y_{m+\frac{1}{2}})=f_{\epsi (m+\frac{1}{2})-\frac{1}{2}}Y_{\epsi (m+\frac{1}{2})}
\ \quad \ \forall\;m\in\Z,
\end{equation}where $\a, \mu, f_{\epsi (m+\frac{1}{2})-\frac{1}{2}}\in\C^{*}, c, d\in\C$.
By $$[\bar{\si}(L_{m}),\bar{\si}(Y_{n+\frac{1}{2}})]=-(n+\dis\frac{1-m}{2})\bar{\si}(Y_{m+n+\frac{1}{2}}), \ \
[\bar{\si}(Y_{m+\frac{1}{2}}),\bar{\si}(Y_{n+\frac{1}{2}})]=(m-n)\bar{\si}(I_{m+n+1}),$$
we have $$f_{\epsi (m+\frac{1}{2})-\frac{1}{2}}=\a^{m}f_{\epsi (0+\frac{1}{2})-\frac{1}{2}}, \ \
\epsi f_{\epsi (m+\frac{1}{2})-\frac{1}{2}}f_{\epsi (n+\frac{1}{2})-\frac{1}{2}}=\a\mu.$$
It is easy to deduce that $$f_{\epsi (m+\frac{1}{2})-\frac{1}{2}}=\a^m\sqrt{\epsi\a\mu}.$$
Therefore
\begin{equation}\label{E4-3}
\bar{\si}(L_{m})=\epsi \a^{m}L_{\epsi m}+\a^{m}(cm+d)I_{\epsi m}, \
\ \bar{\si}(I_{m})=\a^{m}\mu I_{\epsi m},
\end{equation}
\begin{equation}\label{E4-f3}
\bar{\si}(Y_{m+\frac{1}{2}})=\a^m\sqrt{\epsi\a\mu} Y_{\epsi (m+\frac{1}{2})}
\ \quad \ \forall\;m\in\Z.
\end{equation}
Denote by $\bar{\si}(\epsi, \a,  \mu, c, d)$ the automorphism of
$\W^g(0,0)$ satisfying (\ref{E4-3}) and (\ref{E4-f3}), then
\begin{equation}\label{E4-3-1}
\bar{\si}(\epsi_{1}, \a_{1}, \mu_{1}, c_{1}, d_{1})
\bar{\si}(\epsi_{2}, \a_{2}, \mu_{2}, c_{2}, d_{2})
=\bar{\si}(\epsi_{1}\epsi_{2}, \a^{\epsi_{2}}_{1}\a_{2},
\mu_{1}\mu_{2}, c_{1}+\mu_{1}c_{2}, \epsi_{2}d_{1}+\mu_{1}d_{2}),
\end{equation}
and $\bar{\si}(\epsi_{1}, \a_{1}, \mu_{1}, c_{1}, d_{1})
=\bar{\si}(\epsi_{2}, \a_{2}, \mu_{2}, c_{2}, d_{2})$ if and only if
$\epsi_{1}=\epsi_{2}$, $\a_{1}=\a_{2}$, $\mu_{1}= \mu_{2}$,
$c_{1}=c_{2}$, $d_{1}=d_{2}$. Let
$$\bar{\si}_{\epsi}=\bar{\si}(\epsi, 1,  1, 0, 0),\ \
\bar{\tau}_{\a, \mu}=\bar{\si}(1, \a, \mu, 0, 0),\ \ \bar{\si}_{c,
d}=\bar{\si}(1, 1, 1, c, d)$$
$$\mathfrak{a}_{3}=\{\bar{\si}_{\epsi} \;|\; \epsi\in\{\pm 1\} \}, \
\ \mathfrak{b}_{3}=\{\bar{\tau}_{\a, \mu}\;|\;\a, \mu\in\C^{*}\}, \
\ \mathfrak{c}_{3}=\{\bar{\si}_{c, d}\;|\; c,d\in\C\}.$$ From
(\ref{E4-3-1}) we can see that $\mathfrak{a}_{3}$, $
\mathfrak{b}_{3} $  and $ \mathfrak{c}_{3} $ are all subgroups of
$\Aut({\W^g(0,0)})$, and $\mathfrak{a}_{3}\cong\Z_{2},
\mathfrak{b}_{3}\cong\C^{*}\times\C^{*},
\mathfrak{c}_{3}\cong\C\times\C$. Using (\ref{E4-0}), (\ref{E4-f1}) and
(\ref{E4-3-1}), one can deduce that $\mathfrak{I}$ commutates with
$\mathfrak{c}_{3}$. By (\ref{E4-3-1}), we have the following
relations:
$$\bar{\si}_{\epsi}\bar{\tau}_{\a, \mu}=\bar{\si}(\epsi, \a, \mu, 0, 0),\ \
\bar{\tau}_{\a, \mu}\bar{\si}_{c, d}=\bar{\si}(1, \a, \mu, \mu c,\mu
d), \ \ \bar{\si}_{\epsi}\bar{\si}_{c,d}=\bar{\si}(\epsi,1,1,c,d),$$
$$\bar{\tau}_{\a, \mu}\bar{\si}_{\epsi}=\bar{\si}(\epsi, \a^{\epsi}, \mu, 0,0), \ \
\bar{\si}_{c, d}\bar{\tau}_{\a, \mu}=\bar{\si}(1, \a, \mu, c, d), \
\ \bar{\si}_{c,d}\bar{\si}_{\epsi}=\bar{\si}(\epsi, 1, 1, c,\epsi
d),$$
$$\bar{\tau}_{\a, \mu}^{-1}\bar{\si}_{\epsi}\bar{\tau}_{\a, \mu}=\bar{\si}(\epsi, \a^{1-\epsi}, 1, 0, 0), \ \
\bar{\si}_{\epsi}\bar{\tau}_{\a,\mu}\bar{\si}_{\epsi}=\bar{\tau}_{\a^{\epsi},\mu},$$
$$\bar{\si}_{c, d}^{-1}\bar{\tau}_{\a, \mu}\bar{\si}_{c, d}=\bar{\si}(1, \a, \mu, (\mu-1)c,(\mu-1)d),
\ \ \bar{\tau}_{\a, \mu}^{-1}\bar{\si}_{c, d}\bar{\tau}_{\a,
\mu}=\bar{\si}_{\mu^{-1}c, \mu^{-1}d},$$
$$\bar{\si}_{c, d}^{-1}\bar{\si}_{\epsi}\bar{\si}_{c, d}=\bar{\si}(\epsi, 1, 1, 0,
(1-\epsi)d), \ \
\bar{\si}_{\epsi}\bar{\si}_{c,d}\bar{\si}_{\epsi}=\bar{\si}_{c,\epsi
d}.$$ Therefore, $\mathfrak{I}\mathfrak{c}_{3}$ is an abelian normal
subgroup of $\Aut(\W^g(0,0))$. Similar to \cite{SJ}, we have
$$\Aut(\W(0,0))=(\mathfrak{I}\mathfrak{c}_{3})\rtimes(\mathfrak{a}_{3}\ltimes\mathfrak{b}_{3})
\cong (\C^{\infty}\rtimes\C^{\infty})\rtimes(\Z_{2}\ltimes(\C^{*}\times\C^{*})).$$
\\
{\bf Case 2} $b=1$. According to \cite{SCY}, we may assume
\begin{equation}\nonumber
\bar{\si}(L_{m})=\epsi \a^{m}L_{\epsi m}+\a^{m}m[mc+d]I_{\epsi m},
 \ \ \bar{\si}(I_{m})=\a^{m}\mu I_{\epsi m},
\end{equation}
\begin{equation}\nonumber
\bar{\si}(Y_{m+\frac{1}{2}})=f_{\epsi (m+\frac{1}{2})-\frac{1}{2}}Y_{\epsi (m+\frac{1}{2})}, \ \forall\;m\in\Z.
\end{equation}where $\a, \mu, f_{\epsi (m+\frac{1}{2})-\frac{1}{2}}\in\C^{*}, c, d\in\C$.
Similarly, we can deduce that $$f_{\epsi (m+\frac{1}{2})-\frac{1}{2}}=\a^m\sqrt{\epsi\a\mu}.$$
Therefore
\begin{equation}\label{E4-4}
\bar{\si}(L_{m})=\epsi \a^{m}L_{\epsi m}+\a^{m}m(cm+d)I_{\epsi m}, \
\ \bar{\si}(I_{m})=\a^{m}\mu I_{\epsi m},
\end{equation}
\begin{equation}\label{E4-f4}
\bar{\si}(Y_{m+\frac{1}{2}})=\a^m\sqrt{\epsi\a\mu}\ Y_{\epsi (m+\frac{1}{2})}
\ \quad \ \forall\;m\in\Z.
\end{equation}
Denote by $\bar{\si}(\epsi, \a,  \mu, c, d)$ the automorphism of
$\W^g(0,1)$ satisfying (\ref{E4-4}) and (\ref{E4-f4}), then
\begin{equation}\label{E4-4-1}
\bar{\si}(\epsi_{1}, \a_{1}, \mu_{1}, c_{1}, d_{1})
\bar{\si}(\epsi_{2}, \a_{2}, \mu_{2}, c_{2}, d_{2})
=\bar{\si}(\epsi_{1}\epsi_{2}, \a^{\epsi_{2}}_{1}\a_{2},
\mu_{1}\mu_{2}, \epsi_{2}c_{1}+\mu_{1}c_{2}, d_{1}+\mu_{1}d_{2}),
\end{equation}
and $\bar{\si}(\epsi_{1}, \a_{1}, \mu_{1}, c_{1}, d_{1})
=\bar{\si}(\epsi_{2}, \a_{2}, \mu_{2}, c_{2}, d_{2})$ if and only if
$\epsi_{1}=\epsi_{2}$, $\a_{1}=\a_{2}$, $\mu_{1}= \mu_{2}$,
$c_{1}=c_{2}$, $d_{1}=d_{2}$. Let
$$\bar{\si}_{\epsi}=\bar{\si}(\epsi, 1,  1, 0, 0),\ \
\bar{\tau}_{\a, \mu}=\bar{\si}(1, \a, \mu, 0, 0),\ \ \bar{\si}_{c,
d}=\bar{\si}(1, 1, 1, c, d)$$
$$\mathfrak{a}_{4}=\{\bar{\si}_{\epsi} \;|\; \epsi\in\{\pm 1\} \}, \
\ \mathfrak{b}_{4}=\{\bar{\tau}_{\a, \mu}\;|\;\a, \mu\in\C^{*}\}, \
\ \mathfrak{c}_{4}=\{\bar{\si}_{c, d}\;|\; c,d\in\C\}.$$ From
(\ref{E4-4-1}) we can see that $\mathfrak{a}_{4}$, $
\mathfrak{b}_{4} $  and $ \mathfrak{c}_{4} $ are all subgroups of
$\Aut({\W^g(0,1)})$, and $\mathfrak{a}_{4}\cong\Z_{2},
\mathfrak{b}_{4}\cong\C^{*}\times\C^{*},
\mathfrak{c}_{4}\cong\C\times\C$. Similar to the proof above, using
(\ref{E4-0}), (\ref{E4-f1}) and (\ref{E4-4-1}), we also have $\mathfrak{I}$
commutates with $\mathfrak{c}_{4}$. By (\ref{E4-4-1}), we have the
following relations:
$$\bar{\si}_{\epsi}\bar{\tau}_{\a, \mu}=\bar{\si}(\epsi, \a, \mu, 0, 0),\ \
\bar{\tau}_{\a, \mu}\bar{\si}_{c, d}=\bar{\si}(1, \a, \mu, \mu c,\mu
d), \ \ \bar{\si}_{\epsi}\bar{\si}_{c,d}=\bar{\si}(\epsi,1,1,c,d),$$
$$\bar{\tau}_{\a, \mu}\bar{\si}_{\epsi}=\bar{\si}(\epsi, \a^{\epsi}, \mu, 0,0), \ \
\bar{\si}_{c, d}\bar{\tau}_{\a, \mu}=\bar{\si}(1, \a, \mu, c, d), \
\ \bar{\si}_{c,d}\bar{\si}_{\epsi}=\bar{\si}(\epsi, 1, 1, \epsi c,
d),$$
$$\bar{\tau}_{\a, \mu}^{-1}\bar{\si}_{\epsi}\bar{\tau}_{\a, \mu}=\bar{\si}(\epsi, \a^{1-\epsi}, 1, 0, 0), \ \
\bar{\si}_{\epsi}\bar{\tau}_{\a,\mu}\bar{\si}_{\epsi}=\bar{\tau}_{\a^{\epsi},\mu},$$
$$\bar{\si}_{c, d}^{-1}\bar{\tau}_{\a, \mu}\bar{\si}_{c, d}=\bar{\si}(1, \a, \mu, (\mu-1)c,(\mu-1)d),
\ \ \bar{\tau}_{\a, \mu}^{-1}\bar{\si}_{c, d}\bar{\tau}_{\a,
\mu}=\bar{\si}_{\mu^{-1}c, \mu^{-1}d},$$
$$\bar{\si}_{c, d}^{-1}\bar{\si}_{\epsi}\bar{\si}_{c, d}=\bar{\si}(\epsi, 1, 1,
(1-\epsi)c, 0), \ \
\bar{\si}_{\epsi}\bar{\si}_{c,d}\bar{\si}_{\epsi}=\bar{\si}_{\epsi
c, d}.$$ Therefore, $\mathfrak{I}\mathfrak{c}_{4}$ is an abelian
normal subgroup of $\Aut(\W^g(0,1))$ and
$$\Aut(\W^g(0,1))=(\mathfrak{I}\mathfrak{c}_{4})\rtimes(\mathfrak{a}_{4}\ltimes\mathfrak{b}_{4})
\cong (\C^{\infty}\rtimes\C^{\infty})\rtimes(\Z_{2}\ltimes(\C^{*}\times\C^{*})).$$
\\
{\bf Case 3} $b\neq 0, 1$. By \cite{SCY}, we assume
$$\bar{\si}(L_{m})=\epsi \a^{m}L_{\epsi m}+\a^{m}m\l_{\epsi}I_{\epsi
m}, \ \ \bar{\si}(I_{m})=\a^{m}\mu I_{\epsi m},$$
$$\bar{\si}(Y_{m+\frac{1}{2}})=f_{\epsi (m+\frac{1}{2})-\frac{1}{2}}Y_{\epsi (m+\frac{1}{2})}, \ \forall\;m\in\Z.$$
Similarly, we still have $$f_{\epsi (m+\frac{1}{2})-\frac{1}{2}}=\a^m\sqrt{\epsi\a\mu}, \ \ \ \forall m\in\Z.$$
Therefore
\begin{equation}\label{E4-5}
\bar{\si}(L_{m})=\epsi \a^{m}L_{\epsi m}+\a^{m}m(cm+d)I_{\epsi m}, \
\ \bar{\si}(I_{m})=\a^{m}\mu I_{\epsi m},
\end{equation}
\begin{equation}\label{E4-f5}
\bar{\si}(Y_{m+\frac{1}{2}})=\a^m\sqrt{\epsi\a\mu}\ Y_{\epsi (m+\frac{1}{2})}
\ \quad \ \forall\;m\in\Z.
\end{equation}
By  the automorphism group of $\iv(0,-1)$ in \cite{GJP}, we have
$$\Aut(\W^g(0,b))\cong(\C^{\infty}\rtimes\C^{\infty})\rtimes(\Z_{2}\ltimes(\C^{*}\times\C^{*})),$$
where  $b\neq 0,1$.
\end{proof}
\begin{lem}\label{L4.f3}
For $a=1$, $\Aut({\Wgab})\cong(\C^{\infty}\rtimes\C^{\infty})\rtimes(\Z_2\ltimes(\C^{*}\times \C^{*}))$.\\
where  $\C^{\infty}=\{(a_{i})_{i\in\Z}\;|\; a_{i}\in\C, {\rm{\; all
\; but\;  \; a \;  finite \; number \; of \;  the}} \; a_{i} \;
{\rm{ \; are \; zero}} \ \}$, $\Z_{2}=\Z/2\Z$, $\C^{*}=\C\backslash\{0\}$.
\end{lem}
\begin{proof}
For any $\si\in\Aut(\W^g(1,b))$, similar to the ideas in lemma \ref{L4.f1}-\ref{L4.f2}, there exist $\tau\in\mathfrak{I}$ and $\bar{\si}=\tau^{-1}\si$. We may assume that
$$\bar{\si}(L_{0})=\epsi L_{0}+\l_{-1}I_{-1}+l_{-1}Y_{-1+\frac{1}{2}}, \ \bar{\si}(L_{n})=\epsi \a^{n}L_{\epsi n}+\a^{n}\sum\l_{n_i} I_{n_i}
+\a^{n}\sum\mu_{m_j} Y_{m_j+\frac{1}{2}}, \ \forall n\neq 0,$$
$$\bar{\si}(I_{n})=\sum c_p I_{p}+\sum d_q Y_{q+\frac{1}{2}}, \ \bar{\si}(Y_{n+\frac{1}{2}})=\a^{n}\sum e_u I_{u}+\sum f_v Y_{v+\frac{1}{2}}, \ \forall\;n\in\Z,$$
where each $\sum$ formula
is of finite terms and $\l_{n_{i}}, \mu_{m_{j}}, c_{p}, d_{q}, e_{u}, f_{v}\in\C, n_{i}, m_{j}, p,
q, u, v\in\Z$, $\a\in\C^*, \epsi=\pm 1$.
By identities
$$[\bar{\si}(L_0), \bar{\si}(L_n)]=-n\bar{\si}(L_n), \ \ [\bar{\si}(L_0), \bar{\si}(I_n)]=-(n+1)\bar{\si}(I_n),$$
$$[\bar{\si}(L_0), \bar{\si}(Y_{n+\frac{1}{2}})]=-(n+1)\bar{\si}(Y_{n+\frac{1}{2}}),$$
we have $$\mu_{m_j}[\epsi (m_j+1)-n]=0, \ \ \; \l_{-1}bnI_{\epsi n-1}=\sum [\epsi(n_i+1)-n]\l_{n_i}I_{n_i}.$$
$$d_q [\epsi (q+1)-(n+1)]=0, \ \ \epsi\sum c_p(p+1)I_p+\sum l_{-1}d_q(q+1)I_q=(n+1)\sum c_pI_p,$$
$$e_u [\epsi (u+1)-(n+1)]=0, \ \ \ \; f_v [\epsi (v+1)-(n+1)]=0.$$
So  $$m_j=\epsi n-1, \ \l_{-1}b=0, \ n_i=\epsi n-1,$$ $$l_{-1}=0, \ \ p=q=\epsi (n+1)-1, \ \ u=v=\epsi (n+1)-1,$$ and
$$\bar{\si}(L_{0})=\epsi L_{0}+\l_{-1}I_{-1}, \ \ \bar{\si}(L_{n})=\epsi \a^{n}L_{\epsi n}+\a^{n}\l_{\epsi n-1} I_{\epsi n-1}
+\a^{n}\mu_{\epsi n-1} Y_{\epsi n-1+\frac{1}{2}},$$
$$\bar{\si}(I_{n})=c_{\epsi (n+1)-1} I_{\epsi (n+1)-1}+d_{\epsi (n+1)-1} Y_{\epsi (n+1)-1+\frac{1}{2}},$$
$$\bar{\si}(Y_{n+\frac{1}{2}})=\a^{n}e_{\epsi (n+1)-1} I_{\epsi (n+1)-1}+f_{\epsi (n+1)-1} Y_{\epsi (n+1)-1+\frac{1}{2}}.$$
On the other hand, for $m\neq n$, by $[\bar{\si}(I_m), \bar{\si}(I_n)]=0,$
we have $$d_{\epsi (m+1)-1}d_{\epsi (n+1)-1}\epsi (m-n)=0.$$ So $$d_{\epsi (n+1)-1}=0$$
 for all $n\in\Z$. Consequently, by $[\bar{\si}(L_m), \bar{\si}(I_n)]=-(n+1+bm)\bar{\si}(I_{m+n}),$ we have
 $$(n+1+bm)(\a^{m}c_{\epsi (n+1)-1}-c_{\epsi (m+n+1)-1})=0.$$ We can deduce that
 $$c_{\epsi (m+1)-1}=\a^{m}\mu,$$ So \begin{equation}\bar{\si}(I_m)=\a^{m}\mu I_{\epsi (m+1)-1},\ \ \forall m\in\Z\end{equation}
 where $\mu$ is a fixed nonzero complex number.
For $m\neq n$, by $$[\bar{\si}(L_m), \bar{\si}(L_n)]=(m-n)\bar{\si}(L_{m+n}), \ \ [\bar{\si}(L_n), \bar{\si}(Y_{m+\frac{1}{2}})]=-(m+1+\dis\frac{bn-n}{2})\bar{\si}(Y_{m+n+\frac{1}{2}}),$$
$$[\bar{\si}(Y_{m+\frac{1}{2}}), \bar{\si}(Y_{n+\frac{1}{2}})]=(m-n)\bar{\si}(I_{m+n+1}),$$ we have
\begin{equation}\label{ef-1}
\mu_{\epsi m-1}[m+\dis\frac{(b-1)n}{2}]-\mu_{\epsi n-1}[n+\dis\frac{(b-1)m}{2}]=(m-n)\mu_{\epsi (m+n)-1},
\end{equation}
\begin{equation}\label{ef-2}
\mu_{\epsi m-1}\mu_{\epsi n-1}\epsi(m-n)+\l_{\epsi m-1}(m+bn)-\l_{\epsi n-1}(n+bm)=(m-n)\l_{\epsi (m+n)-1},
\end{equation}
\begin{equation}\label{ef-3}
(m+1+\dis\frac{(b-1)n}{2})\a^{m+n}e_{\epsi (m+n+1)-1}-(m+1+b n)\a^{m+n}e_{\epsi (m+1)-1}=\epsi(m+1-n)\a^{n}\mu_{\epsi n-1}f_{\epsi (m+1)-1},
\end{equation}
\begin{equation}\label{ef-4}
\a^{n}f_{\epsi (m+1)-1}(m+1+\dis\frac{(b-1)n}{2})=(m+1+\dis\frac{(b-1)n}{2})f_{\epsi (m+n+1)-1},
\end{equation}
\begin{equation}\label{ef-5}
f_{\epsi (m+1)-1}f_{\epsi (n+1)-1}\epsi=\a^{m+n+1}\mu.
\end{equation}
Let $n=0$ in (\ref{ef-3}) and (\ref{ef-5}), we have
$$ \mu _{-1}=0, \ \  \ f_{\epsi (m+1)-1}=\epsi\a^{m+1}\mu f_{\epsi -1}^{-1}, \ \ \forall m\in\Z,$$
where $f_{\epsi -1}^{-1}=f$ is a fixed nonzero complex number.
From (\ref{ef-4}), we can deduce $f_{\epsi (n+1)-1}=\a^{n}f_{\epsi -1}$ for all $n\in\Z$. Taking it into (\ref{ef-5}), we get
$f_{\epsi -1}=\sqrt{\epsi\a\mu}$, then \begin{equation}\label{ef-10}f_{\epsi (n+1)-1}=\a^{n}\sqrt{\epsi\a\mu}, \ \ \; \forall n\in\Z.\end{equation}
{\bf Case 1:} $b=0$. Let $n=-m, n=1$ in (\ref{ef-1}) respectively, we get
\begin{equation}\label{ef-6}
\mu_{\epsi m-1}+\mu_{\epsi (-m)-1}=0,
\end{equation}
\begin{equation}\label{ef-7}
(2m-1)\mu_{\epsi m-1}-(2-m)\mu_{\epsi-1}=(2m-2)\mu_{\epsi(m+1)-1}, \ \ \forall m\in\Z.
\end{equation}
Using induction on $m\geq 2$ in (\ref{ef-7}), we obtain
\begin{equation}\label{ef-8}
\mu_{\epsi m-1}=\dis\frac{(2m-3)!!}{2^{m-2}(m-2)!}(\mu_{\epsi 2-1}-2\mu_{\epsi -1})+m\mu_{\epsi -1}, \ \ m\geq 2.\end{equation}
On the other hand, let $m=-3, n=1$ in (\ref{ef-1}) and use (\ref{ef-6}), (\ref{ef-8}), we have $\mu_{\epsi 2-1}-2\mu_{\epsi -1}=0$.
So \begin{equation}\label{ef-9}\mu_{\epsi m-1}=m\mu_{\epsi -1}, \ \ \ \forall m\in\Z.\end{equation}
Let $m=-1$ in (\ref{ef-3}) and use (\ref{ef-9}), (\ref{ef-10}), we obtain
\begin{equation}\label{ef-11}e_{\epsi (n+1)-1}=2\epsi (n+1)\mu_{\epsi -1}\sqrt{\epsi\a\mu}, \ \ \forall n\in\Z.
\end{equation}
From (\ref{ef-2}) and (\ref{ef-9}), we have
\begin{equation}\label{ef-12}mn\epsi(m-n)\mu_{\epsi-1}^{2}+m\l_{\epsi m-1}-n\l_{\epsi n-1}=(m-n)\l_{\epsi (m+n)-1}.
\end{equation}
Let $m+n=0$ in (\ref{ef-12}), we have \begin{equation}\label{ef-14}
\l_{\epsi m-1}+\l_{\epsi (-m)-1}=2\l_{-1}+2m^2\epsi\mu_{\epsi-1}^{2}, \ \ \forall m\in\Z.\end{equation}
Let $n=1$ in (\ref{ef-12}) and use induction on m, we obtain
\begin{equation}\label{ef-13}
\l_{\epsi m-1}=(m-1)(\l_{\epsi 2-1}-\l_{\epsi -1})+(m-1)(m-2)\epsi\mu_{\epsi-1}^{2}+\l_{\epsi -1}, \ \ \ m\geq 1.
\end{equation}
On the other hand, let $m=-2, n=1$ and $m=-5, n=3$ in (\ref{ef-12}) respectively and use (\ref{ef-14}), (\ref{ef-13}), then
$$-4\epsi\mu_{\epsi-1}^{2}+2\l_{-1}+2\l_{\epsi 2 -1}-4\l_{\epsi -1}=0,\ \ -4\epsi\mu_{\epsi-1}^{2}+3\l_{-1}+3\l_{\epsi 2 -1}-6\l_{\epsi -1}=0.$$
Combine the two identities, we get $$\l_{-1}=-\l_{\epsi 2 -1}+2\l_{\epsi -1}\ \ \ \; \mbox{and} \ \ \ \ \; \mu_{\epsi-1}=0.$$
So $$\l_{\epsi m-1}=m(\l_{\epsi -1}-\l_{-1})+\l_{-1}, \ \ \ \forall m\in\Z.$$
Set $\l_{\epsi -1}-\l_{-1}=c, \l_{-1}=d$, then
\begin{equation}\label{ef-15}\bar{\si}(L_{n})=\epsi \a^{n}L_{\epsi n}+\a^{n}(nc+d)I_{\epsi n-1},\end{equation}
\begin{equation}\label{ef-16}\bar{\si}(I_n)=\a^{n}\mu I_{\epsi (n+1)-1},\ \ \ \bar{\si}(Y_{n+\frac{1}{2}})=\a^{n}\sqrt{\epsi\a\mu}\ Y_{\epsi (n+1)-1+\frac{1}{2}},\end{equation}
for all $n\in\Z.$ Where $\a, \mu\in\C^*, c, d\in\C, \epsi=\pm{1}$.
Denote by $\bar{\si}(\epsi, \a,  \mu, c, d)$ the automorphism of
$\W^g(1,0)$ satisfying (\ref{ef-15}) and (\ref{ef-16}), then
\begin{equation}\label{ef-17}
\bar{\si}(\epsi_{1}, \a_{1}, \mu_{1}, c_{1}, d_{1})
\bar{\si}(\epsi_{2}, \a_{2}, \mu_{2}, c_{2}, d_{2})
=\bar{\si}(\epsi_{1}\epsi_{2}, \a^{\epsi_{2}}_{1}\a_{2},
\a^{\epsi_2-1}_{1}\mu_{1}\mu_{2}, c_{1}+\a_1^{-1}\mu_{1}c_{2}, \epsi_{2}d_{1}+\a_1^{-1}\mu_{1}d_{2}),
\end{equation}
and $\bar{\si}(\epsi_{1}, \a_{1}, \mu_{1}, c_{1}, d_{1})
=\bar{\si}(\epsi_{2}, \a_{2}, \mu_{2}, c_{2}, d_{2})$ if and only if
$\epsi_{1}=\epsi_{2}$, $\a_{1}=\a_{2}$, $\mu_{1}= \mu_{2}$,
$c_{1}=c_{2}$, $d_{1}=d_{2}$. Let
$$\bar{\si}_{\epsi}=\bar{\si}(\epsi, 1,  1, 0, 0),\ \
\bar{\tau}_{\a, \mu}=\bar{\si}(1, \a, \mu, 0, 0),\ \ \bar{\si}_{c,
d}=\bar{\si}(1, 1, 1, c, d)$$
$$\mathfrak{a}_{3}=\{\bar{\si}_{\epsi} \;|\; \epsi\in\{\pm 1\} \}, \
\ \mathfrak{b}_{3}=\{\bar{\tau}_{\a, \mu}\;|\;\a, \mu\in\C^{*}\}, \
\ \mathfrak{c}_{3}=\{\bar{\si}_{c, d}\;|\; c,d\in\C\}.$$
Similar to the discussion in the lemma \ref{L4.f2}, we have the same results
$$\Aut(\W^g(1,0))=(\mathfrak{I}\mathfrak{c}_{3})\rtimes(\mathfrak{a}_{3}\ltimes\mathfrak{b}_{3})
\cong (\C^{\infty}\rtimes\C^{\infty})\rtimes(\Z_{2}\ltimes(\C^{*}\times\C^{*})).$$
{\bf Case 2:} $b=1$. We have $\l_{-1}=0$, so $\bar{\si}(L_0)=\epsi L_0$. Let $m+n=0$ in (\ref{ef-1}) and (\ref{ef-2}), we have
\begin{eqnarray}\label{ef-18}\mu_{\epsi m-1}+\mu_{\epsi (-m)-1}&=&0,\\
\label{ef-19}
\mu_{\epsi m-1}\mu_{\epsi(-m)-1}&=&0,
\end{eqnarray} for all $m\in\Z$.
From (\ref{ef-18}) and (\ref{ef-19}), we can deduce that $$ \mu_{\epsi m-1}=0, \ \ \forall m\in\Z.$$
Consequently, by (\ref{ef-2}), we have \begin{equation}\label{ef-20}(m+n)(\l_{\epsi m-1}-\l_{\epsi n-1})=(m-n)\l_{\epsi(m+n)-1}\end{equation}
Let $n=-1$ in (\ref{ef-20}) and use induction on $m$, we get
$$\l_{\epsi m-1}=\dis\frac{m(m-1)}{2}(\l_{\epsi-1}+\l_{-\epsi-1})+m\l_{\epsi-1}, \ \ \quad \forall m\in\Z.$$
By (\ref{ef-3}), we have \begin{equation}\label{ef-21}
(m+1)e_{\epsi(m+n+1)-1}=(m+n+1)e_{\epsi(m+1)-1}.\end{equation}
Let $m=0$ in (\ref{ef-21}), we get $$e_{\epsi(n+1)-1}=(n+1)e_{\epsi-1}, \ \ \quad \forall n\in\Z.$$
%By (\ref{ef-4}), we have \begin{equation}\label{ef-22}
%\a^nf_{\epsi(m+1)-1}(m+1)=(m+1)f_{\epsi(m+n+1)-1}.\end{equation}
%Let $m=0$ in (\ref{ef-22}), we still have $f_{\epsi(n+1)-1}=\a^nf_{\epsi-1}$ for all $n\in\Z.$ Consequently, we get the same result as (\ref{ef-10})
%$$f_{\epsi(n+1)-1}=\a^n\sqrt{\epsi\a\mu}, \ \ \quad \forall n\in\Z.$$
Set $c=\dis\frac{\l_{\epsi-1}+\l_{-\epsi-1}}{2},\ d=\l_{\epsi-1}-c,\ e=e_{\epsi-1}$, then
\begin{equation}\label{ef-23}\bar{\si}(L_{n})=\epsi \a^{n}L_{\epsi n}+\a^{n}n(nc+d)I_{\epsi n-1},\ \ \
\bar{\si}(I_n)=\a^{n}\mu I_{\epsi (n+1)-1},\end{equation}
\begin{equation}\label{ef-24}\bar{\si}(Y_{n+\frac{1}{2}})=\a^{n}(n+1)eI_{\epsi (n+1)-1}+\a^n\sqrt{\epsi\a\mu}\ Y_{\epsi (n+1)-1+\frac{1}{2}},\end{equation}for all $n\in\Z$. Where $c,\ d,\ e\in\C,\ \a,\ \mu\in\C^*,\ \epsi=\pm 1$.
If $\bar{\si}$ is a linear operator on $\W^g(a,b)$ satisfying (\ref{ef-23}) and (\ref{ef-24}) for some $\a, \mu\in\C^*$, then it is easy to
check that $\bar{\si}\in\Aut(W^g(a,b))$. Denote by $\bar{\si}(\epsi, \a,  \mu, c, d, e)$ the automorphism of
$\W^g(1,1)$ satisfying (\ref{ef-23}) and (\ref{ef-24}), then
\begin{eqnarray}\label{ef-17}
&&\bar{\si}(\epsi_{1}, \a_{1}, \mu_{1}, c_{1}, d_{1}, e_{1})
\bar{\si}(\epsi_{2}, \a_{2}, \mu_{2}, c_{2}, d_{2}, e_{2})
\\&&=\bar{\si}(\epsi_{1}\epsi_{2}, \a^{\epsi_{2}}_{1}\a_{2},
\a^{\epsi_2-1}_{1}\mu_{1}\mu_{2}, \epsi_{2}c_{1}+\a_1^{-1}\mu_{1}c_{2}, d_{1}+\a_1^{-1}\mu_{1}d_{2}, \a^{\epsi_2-1}_{1}\mu_{1}e_{2}+\a^{\epsi_2-1}_{1}\epsi_{2}e_{1}\sqrt{\epsi_2\a_2\mu_2}),\nonumber
\end{eqnarray}
and $\bar{\si}(\epsi_{1}, \a_{1}, \mu_{1}, c_{1}, d_{1}, e_1)
=\bar{\si}(\epsi_{2}, \a_{2}, \mu_{2}, c_{2}, d_{2}, e_2)$ if and only if
$\epsi_{1}=\epsi_{2}$, $\a_{1}=\a_{2}$, $\mu_{1}= \mu_{2}$,
$c_{1}=c_{2}$, $d_{1}=d_{2}$, $e_1=e_2$. Let
$$\bar{\si}_{\epsi}=\bar{\si}(\epsi, 1,  1, 0, 0, 0),\ \
\bar{\tau}_{\a, \mu}=\bar{\si}(1, \a, \mu, 0, 0, 0),\ \ \bar{\si}_{c,
d, e}=\bar{\si}(1, 1, 1, c, d, e)$$
$$\mathfrak{a}_{3}=\{\bar{\si}_{\epsi} \;|\; \epsi\in\{\pm 1\} \}, \
\ \mathfrak{b}_{3}=\{\bar{\tau}_{\a, \mu}\;|\;\a, \mu\in\C^{*}\}, \
\ \mathfrak{c}_{3}=\{\bar{\si}_{c, d, e}\;|\; c,d,e\in\C\}.$$
From (\ref{ef-17}), we can see that $\mathfrak{a}_{3}$, $\mathfrak{b}_{3}$, $\mathfrak{c}_{3},$ are subgroups of $\Aut(\W^g(1,1))$. Similar to the discussion in lemma \ref{L4.f2} and \cite{SJ}, we have
$$\Aut(\W^g(1,1))=(\mathfrak{I}\mathfrak{c}_{3})\rtimes(\mathfrak{a}_{3}\ltimes\mathfrak{b}_{3})
\cong (\C^{\infty}\rtimes\C^{\infty})\rtimes(\Z_{2}\ltimes(\C^{*}\times\C^{*})).$$
{\bf Case 3:} $b\neq 0,1$. We always have $\l_{-1}=0$. Let $n=-m, n=1$ in (\ref{ef-1}) respectively, we have
\begin{equation}\label{ef-25}(3-b)(\mu_{\epsi m-1}+\mu_{\epsi (-m)-1})=0,
\end{equation}
$$(2m+b-1)\mu_{\epsi m-1}-[2+(b-1)m]\mu_{\epsi -1}=2(m-1)\mu_{\epsi (m+1)-1}.$$
That is \begin{equation}\label{ef-26}2(m-1)[\mu_{\epsi (m+1)-1}-(m+1)\mu_{\epsi -1}]=(2m+b-1)[\mu_{\epsi m-1}-m\mu_{\epsi -1}].\end{equation}
{\bf Subcase 3.1:} $b\neq 3$. From (\ref{ef-25}), we have
\begin{equation}\label{ef-27}\mu_{\epsi (-m)-1}=-\mu_{\epsi m-1}, \quad \forall m\in\Z.\end{equation}
So using induction on $m\geq 2$ in (\ref{ef-26}), we obtain
\begin{equation}\label{ef-28}\mu_{\epsi m-1}=\dis\frac{(3+b)(5+b)\cdots (2m-3+b)}{(2m-4)!!}(\mu_{\epsi 2-1}-2\mu_{\epsi -1})+m\mu_{\epsi -1}, m\geq 3.
\end{equation}
On the other hand, let $m=3, n=-1$ in (\ref{ef-1}) and use (\ref{ef-27}), (\ref{ef-28})
we have
\begin{equation}\label{ef-29}(5+4b-b^2)\mu_{\epsi 2-1}=2(5+4b-b^2)\mu_{\epsi -1}.
\end{equation}
If $5+4b-b^2\neq 0$, that is $b\neq -1$ and $b\neq 5$, then
 $\mu_{\epsi 2-1}=2\mu_{\epsi -1}$.
Consequently, $$\mu_{\epsi m-1}=m\mu_{\epsi -1}, \ \ \forall m\in\Z.$$
From (\ref{ef-2}), we have \begin{equation}\label{ef-30}
mn\epsi (m-n)\mu_{\epsi -1}^2+\l_{\epsi m-1}(m+bn)-\l_{\epsi n-1}(n+bm)=(m-n)\l_{\epsi (n+m)-1}.
\end{equation}
Let $n=-1, m+n=0$ in (\ref{ef-30}) respectively, we get
\begin{equation}\label{ef-33}-m(m+1)\epsi\mu_{\epsi -1}^2+\l_{\epsi m-1}(m-b)-\l_{\epsi(-1)-1}(-1+bm)=(m+1)\l_{\epsi (m-1)-1}
\end{equation}

\begin{equation}\label{ef-31}
(1-b)(\l_{\epsi m-1}+\l_{\epsi (-m)-1})=2\epsi m^2\mu_{\epsi -1}^2\end{equation}
Let $m=2$ in (\ref{ef-33}) and use (\ref{ef-31}), we have
\begin{equation}\label{ef-32}(-4+2b)\epsi\mu_{\epsi -1}^2+(2-b)(1-b)\l_{\epsi 2-1}=(1-b)(4-2b)\l_{\epsi -1}
\end{equation}
If $b\neq 2$, we have \begin{equation}\label{ef-34}2\epsi\mu_{\epsi -1}^2+2(1-b)\l_{\epsi -1}=(1-b)\l_{\epsi 2-1}
\end{equation}
Using (\ref{ef-31}), we obtain \begin{equation}\label{ef-56}\l_{\epsi 2-1}=3\l_{\epsi -1}+\l_{\epsi(-1) -1}.\end{equation} Similarly, we also get
$$\l_{\epsi 3-1}=6\l_{\epsi -1}+3\l_{\epsi(-1) -1}$$
Then induction on $m$ in (\ref{ef-33}) and using (\ref{ef-31}) and (\ref{ef-56}), we have
$$\l_{\epsi m-1}=\dis\frac{m(m+1)}{2}\l_{\epsi -1}+\dis\frac{m(m-1)}{2}\l_{\epsi(-1) -1}=\dis\frac{m(m-1)}{2}[\l_{\epsi -1}+\l_{\epsi(-1)-1}]+m\l_{\epsi -1}, \ \forall m\in\Z.$$
Let $m=0, n=1$ and $m=-1$ in (\ref{ef-3}) respectively, we have \begin{equation}\label{ef-36}
(b+1)(e_{\epsi 2-1}-2e_{\epsi-1})=0\end{equation} and
\begin{equation}\label{ef-35}
\dis\frac{b-1}{2}e_{\epsi n-1}-be_{-1}=-\epsi\a\mu_{\epsi n-1}f_{-1}, \ \ n\neq 0.\end{equation}
From the two identities, we have $e_{-1}=0$ if $b\neq -1$.
So $$e_{\epsi n-1}=\dis\frac{-2\epsi n\a\mu_{\epsi-1}f_{-1}}{b-1}, \ \ \forall n\in\Z.$$
Consequently, by (\ref{ef-10}), we get $$e_{\epsi (n+1)-1}=\dis\frac{-2\epsi (n+1)\mu_{\epsi-1}\sqrt{\epsi\a\mu}}{b-1}, \ \ \forall n\in\Z.$$
Set $\frac{\l_{\epsi-1}+\l_{\epsi(-1)-1}}{2}=c, \l_{\epsi-1}=d$, then by (\ref{ef-31}), $\mu_{\epsi-1}=\sqrt{\epsi(1-b)c}$. So we get
\begin{equation}\label{ef-37}\bar{\si}(L_{n})=\epsi \a^{n}L_{\epsi n}+\a^{n}n((n-1)c+d)I_{\epsi n-1}+\a^{n}n\sqrt{\epsi(1-b)c}Y_{\epsi n-1+\frac{1}{2}},\end{equation}
\begin{equation}\label{ef-39}\bar{\si}(I_n)=\a^{n}\mu I_{\epsi (n+1)-1},\end{equation}
\begin{equation}\label{ef-38}\bar{\si}(Y_{n+\frac{1}{2}})=\a^{n}\dis\frac{-2\epsi (n+1)\sqrt{(1-b)c\a\mu}}{b-1}I_{\epsi (n+1)-1}+\a^n\sqrt{\epsi\a\mu}\ Y_{\epsi (n+1)-1+\frac{1}{2}},\end{equation}for all $n\in\Z$. Where $c,\ d \in\C,\ \a,\ \mu\in\C^*,\ \epsi=\pm 1$.
Similar to the discussion in lemma \ref{L4.f2}, we have $$\Aut(\W^g(1,b))
\cong (\C^{\infty}\rtimes\C^{\infty})\rtimes(\Z_{2}\ltimes(\C^{*}\times\C^{*})), \ \ b\neq 0, 1, 2, 3, -1, 5.$$
If $b=2$, Induction on $m$ in (\ref{ef-33}) and using (\ref{ef-31}), we have
$$\l_{\epsi m-1}=\dis\frac{m(m-1)(m+1)}{6}\l_{\epsi2 -1}-\dis\frac{m(m-1)(m-2)}{6}[\l_{\epsi -1}+\l_{\epsi(-1)-1}]-\dis\frac{m(m-2)(m+2)}{3}\l_{\epsi -1},$$ for all $m\in\Z$. Similar to the above discussion, we have $$\Aut(\W^g(1,2))
\cong (\C^{\infty}\rtimes\C^{\infty})\rtimes(\Z_{2}\ltimes(\C^{*}\times\C^{*})).$$
If $b=-1$, by (\ref{ef-1}), we have \begin{equation}\label{ef-40}\mu_{\epsi m-1}+\mu_{\epsi n-1}=\mu_{\epsi(m+n)-1}.\end{equation}
Let $n=1$ in (\ref{ef-40}) and induction on $m>1$, according to (\ref{ef-27}), we still get $\mu_{\epsi m-1}=m\mu_{\epsi-1}$ for all $m\in\Z.$
By (\ref{ef-35}), we obtain $e_{\epsi n-1}=\epsi\a n\mu_{\epsi-1}f_{-1}+e_{-1}$ for all $n\in\Z.$ We still have the following result
$$\Aut(\W^g(1,-1))
\cong (\C^{\infty}\rtimes\C^{\infty})\rtimes(\Z_{2}\ltimes(\C^{*}\times\C^{*})).$$
If $b=5$, by (\ref{ef-1}), we get $$\mu_{\epsi m-1}=\dis\frac{m(m-1)(m+1)}{6}(\mu_{\epsi2-1}-2\mu_{\epsi-1})+m\mu_{\epsi-1}, \ \ \forall m\in\Z.$$
Let $m+n=0, n=1$ in (\ref{ef-2}) respectively, for all $m\in\Z$, we have
\begin{equation}\label{ef-41}2(\l_{\epsi m-1}+\l_{\epsi (-m)-1})=-\epsi\mu_{\epsi m-1}^2, \end{equation}
\begin{equation}\label{ef-42}\epsi(m-1)\mu_{\epsi m-1}\mu_{\epsi-1}+(m+5)\l_{\epsi m-1}-(1+5m)\l_{\epsi-1}=(m-1)\l_{\epsi(m+1)-1}.\end{equation}
Let $m=\pm 2$ in (\ref{ef-42}) respectively and use (\ref{ef-41}), we obtain
\begin{equation}\label{ef-43}\l_{\epsi 2-1}=-\dis\frac{\epsi(\mu_{\epsi 2-1}-\mu_{\epsi-1})^2}{2}+2\l_{\epsi-1},\end{equation}
\begin{equation}\label{ef-44}
\l_{\epsi 3-1}=\epsi\mu_{\epsi2-1}\mu_{\epsi-1}-\frac{7}{2}\epsi(\mu_{\epsi2-1}-\mu_{\epsi-1})^2+3\l_{\epsi-1}.\end{equation}
On the other hand, let $m=-3$ in (\ref{ef-42}) and use (\ref{ef-41}), (\ref{ef-43}), we have
\begin{equation}\label{ef-45}\l_{\epsi 3-1}=26\epsi\mu_{\epsi2-1}\mu_{\epsi-1}-\frac{43}{2}\epsi\mu_{\epsi-1}^2-8\epsi\mu_{\epsi2-1}^2+3\l_{\epsi-1}.\end{equation}
Combine (\ref{ef-44}) and (\ref{ef-45}), we still get $$\mu_{\epsi2-1}=2\mu_{\epsi-1}.$$ So
$$\mu_{\epsi m-1}=m\mu_{\epsi-1}, \ \ \ \forall m\in\Z.$$ Consequently, we get the complete same results as the above lemma,
that is $$\Aut(\W^g(1,5))
\cong (\C^{\infty}\rtimes\C^{\infty})\rtimes(\Z_{2}\ltimes(\C^{*}\times\C^{*})).$$
{\bf Subcase 3.2:} $b=3$. Letting $n=\pm 1$ in (\ref{ef-1}) respectively and using induction on $m$, we can deduce
\begin{equation}\label{ef-50}\mu_{\epsi m-1}=\dis\frac{m(m-1)}{2}(\mu_{\epsi-1}+\mu_{\epsi(-1)-1})+m\mu_{\epsi-1},  \ \forall m\in\Z.\end{equation}
Let $m+n=0, n=1$ in (\ref{ef-2}) respectively, for all $m\in\Z$, we have
\begin{equation}\label{ef-46}\l_{\epsi m-1}+\l_{\epsi (-m)-1}=\epsi\mu_{\epsi m-1}\mu_{\epsi(-m)-1}, \end{equation}
\begin{equation}\label{ef-47}\epsi(m-1)\mu_{\epsi m-1}\mu_{\epsi-1}+(m+3)\l_{\epsi m-1}-(1+3m)\l_{\epsi-1}=(m-1)\l_{\epsi(m+1)-1}.\end{equation}
Induction on $m>2$ in (\ref{ef-47}) and computation by computer, we obtain
\begin{equation}\label{ef-48}\l_{\epsi m-1}=m\l_{\epsi-1}-\dis\frac{m(m-1)}{2}\epsi\mu_{\epsi-1}^2+\dis\frac{m(m^3+m-2)}{4}\epsi(\mu_{\epsi-1}+\mu_{\epsi(-1)-1})\mu_{\epsi-1}, \  m>2.\end{equation}
At the same time, let $m=-3n$ in (\ref{ef-2}), we have
\begin{equation}\label{ef-49}\l_{\epsi (-2n)-1}=-2\l_{\epsi n-1}+\epsi\mu_{\epsi(-3n)-1}\mu_{\epsi n-1}, \ \ \forall n\in\Z.\end{equation}
Let $n=-2$ and $n=1$ in (\ref{ef-49}) and use (\ref{ef-50}), we have
\begin{equation}\label{ef-51}\l_{\epsi 4-1}=4\l_{\epsi-1}+15\epsi\mu_{\epsi-1}^2+45\epsi\mu_{\epsi(-1)-1}^2+66\epsi\mu_{\epsi-1}\mu_{\epsi(-1)-1}.\end{equation}
On the other hand, by (\ref{ef-48}), we have
\begin{equation}\label{ef-52}\l_{\epsi 4-1}=4\l_{\epsi-1}+60\epsi\mu_{\epsi-1}^2+66\epsi\mu_{\epsi-1}\mu_{\epsi(-1)-1}.\end{equation}
Compare (\ref{ef-51}) with (\ref{ef-52}), we obtain $$\mu_{\epsi-1}^2=\mu_{\epsi(-1)-1}^2$$
If $\mu_{\epsi-1}=-\mu_{\epsi(-1)-1}$, by (\ref{ef-50}), (\ref{ef-46}) and (\ref{ef-48}), we have
$$\mu_{\epsi m-1}=m\mu_{\epsi-1}, \ \ \l_{\epsi m-1}=m\l_{\epsi-1}+\dis\frac{m(m-1)}{2}(\l_{\epsi-1}+\l_{\epsi(-1)-1}), \ \forall m\in\Z.$$
Similar to the discussion in lemma \ref{L4.f2}, we get
$$\Aut(\W^g(1,3))
\cong (\C^{\infty}\rtimes\C^{\infty})\rtimes(\Z_{2}\ltimes(\C^{*}\times\C^{*})).$$

If $\mu_{\epsi-1}=\mu_{\epsi(-1)-1}$, by (\ref{ef-50}), (\ref{ef-46}) and (\ref{ef-48}), we have
$$\mu_{\epsi m-1}=m^2\mu_{\epsi-1}, \ \ \l_{\epsi m-1}=m\l_{\epsi-1}+\dis\frac{m(m^3-1)}{2}(\l_{\epsi-1}+\l_{\epsi(-1)-1}), \ \forall m\in\Z.$$
Set $\dis\frac{\l_{\epsi-1}+\l_{\epsi(-1)-1}}{2}=c, \dis\frac{\l_{\epsi-1}-\l_{\epsi(-1)-1}}{2}=d$, then $\mu_{\epsi-1}=\sqrt{2c\epsi}$. Consequently, we have
\begin{equation}\label{ef-53}\bar{\si}(L_{n})=\epsi \a^{n}L_{\epsi n}+\a^{n}n(n^3c+d)I_{\epsi n-1}+\a^{n}n^2\sqrt{2\epsi c}Y_{\epsi n-1+\frac{1}{2}},\end{equation}
\begin{equation}\label{ef-54}\bar{\si}(I_n)=\a^{n}\mu I_{\epsi (n+1)-1},\end{equation}
\begin{equation}\label{ef-55}\bar{\si}(Y_{n+\frac{1}{2}})=-\a^{n}\epsi (n+1)\sqrt{2c\a\mu}I_{\epsi (n+1)-1}+\a^n\sqrt{\epsi\a\mu}\ Y_{\epsi (n+1)-1+\frac{1}{2}},\end{equation}for all $n\in\Z$, where $c,\ d \in\C,\ \a,\ \mu\in\C^*,\ \epsi=\pm 1$.
Similar to the discussion in lemma \ref{L4.f2}, we still have
$$\Aut(\W^g(1,3))
\cong (\C^{\infty}\rtimes\C^{\infty})\rtimes(\Z_{2}\ltimes(\C^{*}\times\C^{*})).$$
\end{proof}
By lemma \ref{L4.f1}-\ref{L4.f3}, we get the main theorem of this section.
\begin{theorem}\label{T4}
\begin{eqnarray}\nonumber
\Aut(\Wgab)\cong\left\{
\begin{array}{ll}
& (\C^{\infty}\rtimes\C^{\infty})\rtimes(\C^{*}\times \C^{*}), \;\;\; if\;
a\not\in\Z,
\vs{4pt}\\
& (\C^{\infty}\rtimes\C^{\infty})\rtimes(\Z_{2}\ltimes(\C^{*}\times\C^{*})),\ \
otherwise.
\end{array}
\right.
\end{eqnarray}
where  $\C^{\infty}=\{(a_{i})_{i\in\Z}\;|\; a_{i}\in\C, {\rm{\; all
\; but\;  \; a \;  finite \; number \; of \;  the}} \; a_{i} \;
{\rm{ \; are \; zero}} \ \}$, and $\Z_{2}=\Z/2\Z$,
$\C^{*}=\C\backslash\{0\}$.
\end{theorem}
\qed

\bibliography{}

\end{document}